\documentclass[12pt]{article}

\usepackage{amsmath}
\usepackage{amssymb}
\usepackage{amsfonts}
\usepackage{amsthm}
\usepackage{a4}
\usepackage{graphicx}
\usepackage[abs]{overpic}
\usepackage{caption}
\usepackage{wrapfig}
\usepackage{float}
\usepackage{graphicx}
\usepackage{bm}
\usepackage{verbatim}

\newtheorem{theorem}{Theorem}[section]
\newtheorem{lemma}[theorem]{Lemma}

\newtheorem{corollary}[theorem]{Corollary}

\theoremstyle{definition}
\newtheorem{rem}[theorem]{Remark}
\newtheorem*{notation}{Notation}

\newcommand{\N}{\Bbb N}

\newcommand{\R}{\Bbb R}

\def\Id{\mathbf{Id}}

\def\eps{\varepsilon}

\def\e{\mathbf{e}}

\def\SO{\mathrm{SO}}
\def\PW{P\mbox{-}W}
\def\weakly{\rightharpoonup}
\def\weaklystar{\stackrel{*}{\rightharpoonup}}
\def\dist{\operatorname{dist}}

\usepackage{color}

\begin{document}

\begin{center}
\begin{Large}
{\bf {A Griffith-Euler-Bernoulli theory for thin \\ brittle beams derived from nonlinear models \\ in variational fracture mechanics \\ }}
\end{Large}
\end{center}

\begin{center}
\begin{large}
Bernd Schmidt\footnote{Universit{\"a}t Augsburg, Institut f{\"u}r Mathematik, 
Universit{\"a}tsstr.\ 14, 86159 Augsburg, Germany. {\tt bernd.schmidt@math.uni-augsburg.de}}\\
\end{large}
\end{center}

\begin{center}
\today
\end{center}
\bigskip

\begin{abstract}
We study a planar thin brittle beam subject to elastic deformations and cracks described in terms of a nonlinear Griffith energy functional acting on $SBV$ deformations of the beam. In particular we consider the case in which elastic bulk contributions due to finite bending of the beam are comparable to the surface energy which is necessary to completely break the beam into several large pieces. In the limit of vanishing aspect ratio we rigorously derive an effective Griffith-Euler-Bernoulli functional which acts on piecewise $W^{2,2}$ regular curves representing the midline of the beam. The elastic part of this functional is the classical Euler-Bernoulli functional for thin beams in the bending dominated regime in terms of the curve's curvature. In addition there also emerges a fracture term proportional to the number of discontinuities of the curve and its first derivative. 
\end{abstract}
\bigskip

\begin{small}

\medskip

\end{small}

\tableofcontents

\section{Introduction}\label{sec:Intro}

The derivation of effective theories for thin structures such as beams, rods, plates and shells is a classical problem in continuum mechanics. In a fundamental contribution Euler proposed an elastic energy functional for a planar thin beam undergoing pure bending so that its midline remains unstretched, in which the effective local energy contributions are proportional to the squared curvature of the deformed midline, cf.\ \cite{Euler}. Basic results in formulating adequate dimensionally reduced theories for three-dimensional elastic objects go back to the work of Kirchhoff and von K{\'a}rm{\'a}n \cite{Kirchhoff,vonKarman}, cf.\ also \cite{Love,Ciarlet-ii,Ciarlet-iii}. First rigorous results, however, deriving effective energy functionals for elastic thin films have been obtained only recently, cf.\ \cite{anzbalper, LeDretRaoult-i, LeDretRaoult-ii, FJM:02, FJM:06}. 

In (nonlinear) elasticity theory, a (hyper-)elastic specimen occupying a region $\Omega \subset \R^d$ ($d = 2,3$) and subject to a deformation $v : \Omega \to \R^d$ is described in terms of its stored energy 
$$ \int_{\Omega} W(\nabla v(x)) \, dx, $$ 
where the stored energy function $W$ acts on the deformation gradient $\nabla v$ measuring the local strain of $v$. For thin structures $\Omega = \Omega_h$ of small aspect ratio $h \ll 1$ such as beams of height $h \ll 1$, three-dimensional rods of thickness $h \ll 1$ or plates of height $h \ll 1$ (with the other dimensions of order $1$) the basic task is to obtain dimensionally reduced energy functionals acting on suitable strain quantities for the limiting one- or two-dimensional objects. In order to rigorously relate these theories to the parent nonlinear bulk elasticity model, one aims at establishing a variational convergence result in the sense of $\Gamma$-convergence, which in particular guarantees that sequences of (almost) minimizers subject to suitable external forces converge to the solution of the effective limiting minimum problem, cf.\ e.g.\ \cite{dalmaso}. Notably the seminal articles \cite{FJM:02, FJM:06}, in which a whole hierarchy (in terms of possible energy scalings) of plate theories is derived from three-dimensional nonlinear elasticity theory, have triggered a still ongoing activity in extending these results to various different settings including effective theories for rods \cite{fjmmii}, shells \cite{fjmm}, atomistic films \cite{Schmidt:06}, heterogeneous layers \cite{Schmidt:07}, incompressible plates \cite{ContiDolzmann:09} and plates with pre-strain \cite{LewickaMahadevanPakzad}. 

When examining thin structures made of brittle materials it is indispensable to consider models beyond the purely elastic regime which, in particular, include the possibility that the body undergoes fracture. Motivated by the pioneering work of Griffith \cite{Griffith:1921}, in which the formation of crack is viewed as the result of a competition between the surface energy cost and the reduction of bulk energy during an infinitesimal increase of the crack set, Francfort and Marigo \cite{Francfort-Marigo:1998} have introduced energy functionals comprising both bulk and surface contributions which lend themselves to a variational analysis. In contrast to the elastic case the deformation $v : \Omega \to \R^d$ may now contain jump discontinuities along a `crack set' of codimension one. In its basic form, when the crack energy is homogeneous, independent of the crack opening and isotropic, a Griffith functional is given by  
\begin{align}\label{eq:Griffith-bulk} 
  \int_{\Omega \setminus J_{v}} W(\nabla v(x)) \, dx + \beta {\cal H}^{d-1} (J_{v}), 
\end{align}
where $\nabla v$ is the bulk deformation gradient, $J_v$ the crack set and $\beta > 0$ a material constant. The crack energy is then simply proportional to the Hausdorff area of $J_v$. 

While of considerable interest both from a theoretical and a practical point of view, it seems that most of the possible energy scaling regimes for thin brittle specimens are yet poorly understood. The problem here is to consider nonlinear Griffith functionals for thin reference configurations $\Omega = \Omega_h$ and to develop dimension reduction techniques which allow for the derivation of a suitable `Griffith plate theory' in the limit $h \to 0$, where the resulting elastic plate theory is augmented with an effective surface term. Notable contributions to this aim are the works of Braides and Fonseca \cite{BraidesFonseca:01} on a $\Gamma$-convergence result and of Babadjian \cite{Babadjian} on a $\Gamma$-convergence and asymptotic quasistatic evolution result in the membrane energy regime, i.e., for deformations with finite energy per unit volume. 

Yet, brittle materials that respond elastically only to very small strains and do not have a significant plastic regime but rather develop cracks already for moderate strains are naturally investigated in the small displacement regime by linearized Griffith functionals whose elastic bulk contribution is the corresponding energy functional of infinitesimal elasticity, cf.\ \cite{Bourdin-Francfort-Marigo:2008}. However, in the presence of cracks a body can break apart into several pieces each of which may afterwards undergo a different rigid motion at no extra energy cost, so that in general it is not possible to linearize a deformation around a single rigid motion. To overcome the serious drawback of linearized functionals not being (nonlinearly) frame invariant, one is therefore led to consider nonlinear Griffith models in which a suitable scaling parameter to either the bulk or the surface part is introduced in order to relate the stiffness and toughness parameters of the material in such a way that the elastic bulk energy for small displacements is of the same order of magnitude as the surface part. In \eqref{eq:Griffith-bulk} this amounts to viewing $\beta$ a small parameter and considering deformations whose Green-St.\ Venant strain tensor is of the order $\sqrt{\beta}$. The problem of deriving linearized theories (around a piecewise rigid motion) from such rescaled nonlinear Griffith functionals has only recently been resolved in a planar setting by Friedrich in \cite{Friedrich:15b}. (See also \cite{NegriToader:2013} for a similar problem within a restricted class of admissible crack sets.) We also remark that this scaling parameter in atomistic systems can be related to the typical (small) interatomic distances, see \cite{Braides-Lew-Ortiz:06, FriedrichSchmidt:14, FriedrichSchmidt:15a, FriedrichSchmidt:15c}. For thin films the inadequacy of a linearized parent model may even occur in the absence of cracks as these objects can be largely deformed by bending with only small energetic cost. As such models already contain the aspect ratio $h$ as a small scaling parameter, this leads to functionals in which $\beta$ is considered a second small parameter. 

In this article we focus on the case of a two-dimensional thin strip $\Omega_h = (0, L) \times (-\tfrac{h}{2}, \tfrac{h}{2})$ subject to a deformation $v : \Omega_h \to \R^2$ of energy 
$$ \int_{\Omega_h} W(\nabla v) \, dx + \beta_h {\cal H}^1 (J_{v}). $$ 
Being interested in beams whose unfractured regions are deformed within the bending dominated regime, we also suppose that $\beta_h \sim h^2$. Our main result in Theorem \ref{theo:Gamma-convergence} is that, under suitable assumptions, in the limit $h \to 0$ these functionals, divided by $h^3$, $\Gamma$-converge to an effective `Griffith-Euler-Bernoulli' functional of the form 
$$ \frac{\alpha}{24} \int_0^L | \kappa(t) |^2 \, dt + \beta \# (J_{\bar{y}} \cup J_{\bar{y}'}), $$ 
where $\bar{y} : (0,L) \to \R^2$ is a piecewise $W^{2,2}$ regular curve (representing the midline of the beam), $\kappa$ is the curvature of $\bar{y}$ and $J_{\bar{y}} \cup J_{\bar{y}'}$ is the set of discontinuities of $\bar{y}$ and $\bar{y}'$. Moreover, $\alpha$ is the Euler-Bernoulli constant and $\beta$ the effective energy per crack. This is complemented with a compactness result on finite energy sequences in Theorem \ref{theo:compactness}. In Corollaries \ref{cor:Gamma-compactness-forces} and \ref{cor:Gamma-compactness-forces-bv} we also show that body forces and boundary conditions can be included in our analysis. As a direct consequence we obtain a convergence result for (almost) minimizers in Corollary \ref{cor:conv-almost-min}. 

In all of the aforementioned results on plate theories with energy scalings beyond the membrane energy regime, at the core of their derivation lies a quantitative geometric rigidity theorem which allows for controlling the deviation of a deformation from a rigid motion in terms its energy. For Sobolev functions such a result has been proven by Friesecke, James and M{\"u}ller in \cite{FJM:02}. In \cite{FJM:02} and \cite{FJM:06} these authors then show how this estimate allows for the derivation of effective plate theories from three-dimensional nonlinear elasticity. In fact, many of the subsequent variations in different settings have either used a functional in \cite{FJM:02,FJM:06} as a comparison functional or followed the same strategy to use the geometric rigidity theorem of \cite{FJM:02} in order to estimate the local deviations from approximating rigid motions. Our approach to brittle thin beams follows the same general strategy: We begin by covering the beam with $O(h^{-1})$ many small rectangles, approximate the deformation by a rigid motion and then compare these rigid motions on neighboring rectangles. 

At the core of our derivation now lies a novel quantitative piecewise geometric rigidity theorem for $SBV$ functions in two dimensions, recently obtained by Friedrich and the author in \cite{FriedrichSchmidt:15b}, cf.\ Theorem \ref{theo:quant-rig}. Due to the possible presence of jump discontinuities this theorem, however, is formulated in a considerably more complex way than the corresponding rigidity result for Sobolev functions. For a given deformation, instead of a single rotation there is an underlying Caccioppoli partition of rotations and only slightly modified versions of the deformations are proved to be close to the corresponding piecewise rigid motions. Moreover, while the piecewise linearized strain can be controlled in terms of the deformation energy, a similar bound on the full gradient is not available. (This is related to the fact that there is no analogue of Korn's inequality in $SBV$.) These notable differences result in a number of technical challenges, including the following.  


1.\ The estimates in Theorem \ref{theo:quant-rig}, which are to be applied on many small rectangles, only apply to a modified deformation, which might considerably deviate from the original one on a set whose smallness is quantified by a small parameter $\rho$, while the constants in these estimates in turn do depend on $\rho$.

2.\ The modified deformations on different rectangles have to be joined together before taking the limit $h \to 0$ in order to obtain the limiting strain of the thin beam. Neither a simple piecewise constant interpolation nor mollified versions thereof are adequate as the former would introduce by far too much artificial crack and the latter will naturally give too weak estimates in the presence of cracks. A subtle point here is that our interpolating and taking gradients do not commute and considerable efforts have to be made to estimate the difference of the interpolated gradients and the gradient of the interpolation.

3.\ An essential step is to show that the limiting strain is asymptotically linear in the vertical small beam direction. To this end, similarly as in \cite{FJM:02} we consider difference quotients in this direction. However, in our setting of $SBV$ functions, the arguments in \cite{FJM:02} do not apply to determining their limiting behavior due the the possible presence of a singular part of the derivative which does not store elastic energy. We propose a different method here by applying an SBV closure argument to an auxiliary function which arises from flattening the beam's deformation and a suitable rescaling of both its image and preimage.

An aspect of our derivation which appears to be interesting also from a physical point of view is that fracture in a limiting deformation can occur only at those points $t \in (0, L)$ for which there are approximating deformations at finite $h$ containing a crack set of length $h$ concentrated in a region converging to $t$ whose $x_1$-projection is much smaller than $h$. Smaller cracks which are separated at least a distance comparable to $h$, on the other hand, are healed in the limit $h \to 0$, cf.\ Remark \ref{rem:crack-localization} for a precise statement. 

We finally remark that the main reason for our restricting to planar beams is that the basic Theorem \ref{theo:quant-rig} is only available in two dimensions. Although we have taken advantage of the possibility to simplify various arguments by exploiting the planar set-up, we believe that our analysis, in particular the aforementioned technical considerations, will essentially allow for the derivation of a Griffith plate theory, provided Theorem \ref{theo:quant-rig} can be extended to three dimensions. 
\pagebreak[3]

\begin{notation} 
For vectors $a \in \R^m$, $b \in \R^d$ we write $a \otimes b$ for the matrix $a b^T \in \R^{m \times d}$. If $a = (a_1, a_2)^T \in \R^2$ we set $a^{\perp} = (-a_2, a_1)^T$. The standard unit vectors in $\R^2$ are $\e_1 = (1,0)^T$ and $\e_2 = (0,1)^T$. If $a, b \in \R^2$ we write $(a \mid b) = a \otimes \e_1 + b \otimes \e_2$. By $\R^{n \times n}_{\rm sym}$ and $\R^{n \times n}_{\rm skew}$ we denote the space of symmetric and the space of skew symmetric $n \times n$ square matrices respectively. The symmetric part of $X \in \R^{n \times n}$ is $e(X) = \frac{X+X^T}{2}$. 

In the proofs in Sections \ref{sec:compactness} and \ref{sec:Gamma} we will encounter the parameters $h \searrow 0, \rho \searrow 0, \lambda \nearrow 1$ and $n \nearrow \infty$ (converging in this order). Generic constants which are independent of $h$ but may depend on $n$ are denoted $C, \hat{C}$ with the convention that $C$ is independent of $\rho$ and $\lambda$ while $\hat{C}$ might depend on $\rho, \lambda$. 
\end{notation}

\section{Main results}

We consider a thin brittle beam whose reference configuration occupies the region $\Omega_h = (0, L) \times (-\tfrac{h}{2}, \tfrac{h}{2})$ in $\R^2$. Here the length $L>0$ of the strip is supposed to be of order $1$ while the beam thickness $h>0$ is assumed to be much smaller than $1$. We fix a (large) constant $M > 0$ and consider deformations $v \in SBV(\Omega_h, \R^2)$ of the beam with $\max \{ \| v \|_{L^{\infty}}, \| \nabla v \|_{L^{\infty}} \} \le M$. (See Section \ref{sec:Preliminaries} for the definition and some essential properties of the space $SBV$.) The constant $M$, assumed to be much larger than $1$ and $L$, effectively confines the specimen to a large box and also forbids (unphysically) large elastic strains. (A more thorough discussion on the restrictions caused by $M$ is given in Remark \ref{rem:main-res}.\ref{rem:main-res-M}.) We assume that the energy $E^h(v)$ of such a deformation consists of two parts: the stored elastic energy which is given as the integral of a stored energy function $W$ evaluated at the absolutely continuous part $\nabla v$ of $D v$ and an isotropic crack energy which is proportional to the length ${\cal H}^1(J_v)$ of the jump set $J_v$ of $v$, i.e.,  
\begin{align*}
  E^h(v) 
  = \int_{\Omega_h} W(\nabla v) \, dx + \beta_h {\cal H}^1 (J_{v}).  
\end{align*} 

Our main assumptions on $W : \R^{2 \times 2} \to [0, \infty)$ are the following: 
\begin{itemize}
\item[(i)] Regularity: $W$ is continuous on $\R^{2 \times 2}$ and $C^3$ in a neighborhood of $\SO(2)$.  
\item[(ii)] Frame indifference: $W(RX) = W(X)$ for all $R \in \SO(2)$ and $X \in \R^{2 \times 2}$. 
\item[(iii)] Non-degeneracy: $W(X) = 0$ if an only if $X \in \SO(2)$ and there is a constant $c > 0$ such that, for all $X \in \R^{2 \times 2}$, 
\begin{align}\label{eq:W-lower-bound} 
  W(X) \ge c \dist^2(X, \SO(2)). 
\end{align}
\end{itemize}

In the absence of cracks, i.e., in the purely elastic setting, bending dominated deformations will store an elastic bending energy scaling with $h^3$, cf.\ \cite{Euler,FJM:02}. On the other hand, the crack length which is necessary to completely break $\Omega_h$ vertically into several pieces is of the order $h$. In order to obtain an energy functional which models brittle beams undergoing fracture within the bending dominated regime and thus, as discussed in the introduction above, account for both energy contributions on the same scale, we will henceforth assume that $\beta_h = \beta h^2$ for a constant $\beta > 0$. 

In order to compare deformations $v$ on $\Omega_h$ for different $h$ we rescale by setting $\Omega = \Omega_1$, $y(x_1, x_2) = v(x_1, h x_2)$ so that $(\partial_1 y \mid h^{-1} \partial_2 y)(x_1, x_2) = \nabla v(x_1, h x_2)$ and abbreviate $(\partial_1 y \mid h^{-1} \partial_2 y) = \nabla_h y$. (We write $(a \mid b)$ for the $2 \times 2$ with columns $a, b \in \R^2$.) We also divide $E^h$ by $h^3$ so as to obtain an energy functional $I^h$ of order $1$ so that 
\begin{align*}
  E^h(v) 
  &= \int_{\Omega_h} W(\nabla v) \, dx + \beta h^2 {\cal H}^1 (J_{v}) \\ 
  &= h \int_{\Omega} W(\nabla_h y) \, dx + \beta h^2 \int_{\Omega \cap J_{y}} |(h \nu_1(y), \nu_2(y))| \, d{\cal H}^1 
   = h^3 I^h(y) 
\end{align*} 
with $\nu(y)$ denoting the crack normal of $J_{y}$, where $I^h : SBV(\Omega; \R^2) \to [0, +\infty]$ is given by 
\begin{align*}
  I^h(y) 
  = \begin{cases} 
       h^{-2} \int_{\Omega} W(\nabla_h y) \, dx + \beta \int_{\Omega \cap J_{y}} |(\nu_1(y), h^{-1} \nu_2(y))| \, d{\cal H}^1 & \mbox{for } y \in {\cal A}^h, \\ 
       + \infty & \mbox{otherwise}
    \end{cases}  
\end{align*} 
for ${\cal A}^h = \big \{ y \in SBV(\Omega; \R^2) : \max \{ \| y \|_{L^{\infty}}, \| \nabla_h y \|_{L^{\infty}} \} \le M \big\}$. 

Let ${\cal Q}$ be the Hessian of $W$ at $\Id$. Note that by the assumptions on $W$, the quadratic form ${\cal Q}$ is positive definite on the space $\R^{2 \times 2}_{\rm sym}$ of symmetric matrices and vanishes on the space $\R^{2 \times 2}_{\rm skew}$ of skew symmetric matrices. We define a relaxed elastic constant by 
\begin{align}\label{eq:Q-def}
  \alpha 
  = \min_{\gamma \in \R^2} {\cal Q} (\mathbf{e}_1 \otimes \mathbf{e}_1 +  \gamma \otimes \mathbf{e}_2) 
  = \min_{\gamma \in \R^2} {\cal Q} \Big( \mathbf{e}_1 \otimes \mathbf{e}_1 + \frac{\gamma \otimes \mathbf{e}_2 + \mathbf{e}_2 \otimes \gamma}{2} \Big), 
\end{align} 
where $\mathbf{e}_1 = (1,0)^T$ and $\mathbf{e}_2 = (0,1)^T$ (and $a \otimes b = a b^T$ for vectors $a, b$). 

Our main results are the following theorems on $\Gamma$-convergence and compactness for $I^h$ as $h \to 0$. 
The limiting deformations $y$ turn out to be independent of the vertical component $x_2$ and the limiting functional $I^0 : SBV(\Omega; \R^2) \to [0, +\infty]$ is a `Griffith-Euler-Bernoulli energy functional' given by 
\begin{align*}
  I^0(y) 
  = \begin{cases} 
       \frac{\alpha}{24} \int_0^L | \kappa(t) |^2 \, dt + \beta \# (J_{\bar{y}} \cup J_{\bar{y}'}) & \mbox{for } y \in {\cal A},~ y(x) = \bar{y}(x_1) \mbox{ a.e., } \\ 
       + \infty & \mbox{otherwise.}
    \end{cases}
\end{align*}
Here the set of admissible limiting deformations is 
\begin{align*}
  {\cal A} 
  &= \big\{ y \in SBV(\Omega; \R^2) : y(x) = \bar{y}(x_1) \mbox{ for a.e.\ } x \in \Omega \\ 
  &\qquad\qquad\qquad \mbox{ with } \bar{y} \in \PW^{2,2}((0, L); \R^2),~ |\bar{y}| \le M \mbox{ and } |\bar{y}'| = 1 \mbox{ a.e.} \big\} 
\end{align*} 
and $\kappa(t) = \bar{y}'' \cdot (\bar{y}')^{\perp}$ is the curvature of the curve $t \mapsto \bar{y}(t)$; see Section \ref{sec:Preliminaries} for the definition of the piecewise Sobolev space $\PW^{2,2}$. 
(By abuse of notation, functions $f$ defined on (a part of) $\Omega$ that only depend on $x_1$ will not be renamed in the sequel so that we often simply write $f(x_1)$ instead of $f(x)$ and $f'$ instead of $\partial_1 f$.)

\begin{theorem}[Gamma-convergence]\label{theo:Gamma-convergence} 
The functionals $I^h$ $\Gamma$-converge to $I^0$ on $SBV(\Omega; \R^2)$ with respect to the strong $L^1$-topology as $h \to 0$, i.e., 
\begin{itemize}
\item[(i)] $\liminf$ inequality: whenever $y^h \to y$ in $L^1$ for $y^h, y \in SBV(\Omega; \R^2)$, 
$$ \liminf_{h \to 0} I^h(y^h) \ge I^0(y); $$ 

\item[(ii)] recovery sequences: for every $y \in SBV(\Omega; \R^2)$ there exist $y^h  \in SBV(\Omega; \R^2)$ with $y^h \to y$ in $L^1$ and 
$$ \lim_{h \to 0} I^h(y^h) = I^0(y). $$
\end{itemize}
\end{theorem} 
We will prove Theorem \ref{theo:Gamma-convergence} in Section \ref{sec:Gamma}. Here the choice of the $L^1$ topology has been made for definiteness. In fact, $\Gamma$-convergence also holds for other choices as will be detailed below. 

This $\Gamma$-convergence theorem is complemented by the following strong compactness result, proved in Section \ref{sec:compactness}. 
\begin{theorem}[compactness]\label{theo:compactness}
Suppose $y^h \in {\cal A}^h$ satisfy 
\begin{align*}
  I^h(y^h) \le C 
\end{align*}
for a constant $C$ independent of $h$. Then there exists a subsequence (not relabeled) verifying the following assertions as $h \to 0$. 
\begin{itemize}
\item[(i)] $y^h \to y$ in $L^1$ for a limiting deformation $y \in {\cal A}$. 
\item[(ii)] The rescaled absolutely continuous part of the gradient satisfies $\nabla_h y^h \to \big( \partial_1 y \mid (\partial_1 y)^{\perp} \big)$ strongly in $L^2(\Omega)$.
\end{itemize}
\end{theorem} 

\begin{rem}\label{rem:main-res}
\begin{enumerate}
\item\label{rem:main-res-M} The constant $M$, assumed to be (much) larger than $1$ and $L$, imposes a restriction (for the unrescaled deformations) on both $\| v \|_{L^{\infty}}$ and $\| \nabla v \|_{L^{\infty}}$. Confining $v$ in this way effectively models a large box containing the deformed specimen $v(\Omega_h)$ and prevents parts of $\Omega_h$ from escaping to $\infty$. The restriction on $\nabla v$ is necessary for technical reasons as it allows us to apply a quantitative piecewise rigidity estimate, recently obtained in \cite{FriedrichSchmidt:15b}. As $M$ can be chosen arbitrarily large, well beyond the elastic regime of the specimen under investigation, in applications such a restriction is in fact not too severe. This is even less so in our present setting of a beam in the bending dominated energy regime where, as we will see, the nonlinear strain is infinitesimally close to $\Id$ and there is no restriction on the linearized (infinitesimal) theory. We also remark that such a constraint on $\nabla v$ can be justified for certain (interpolations of) atomistic models whose small energy is related to the typical interatomic spacing as, e.g., in \cite{FriedrichSchmidt:14,FriedrichSchmidt:15a,FriedrichSchmidt:15c}. In these discrete models an atomistic unit cell can be considered effectively broken whenever its discrete gradient exceeds a finite threshold value. Its contribution to the total energy then enters through the surface part of the energy functional. 

\item\label{rem:main-res-better-top} In view of the $L^{\infty}$-bound on $y^h$ and $\nabla_h y^h$ we have $y^h \to y$ and $\nabla_hy^h \to \big( \partial_1 y \mid (\partial_1 y)^{\perp} \big)$ in $L^p$ for any $1 \le p < \infty$ in Theorem \ref{theo:compactness}. We thus may replace the $L^1$ convergence in Theorem \ref{theo:Gamma-convergence} by convergence in $SBV^p(\Omega; \R^2)$ for any $1 \le p < \infty$ in the strong sense that $y^{h} \to y$, $\nabla y^h \to \nabla y$ strongly in $L^p$ and $D^s y^h \weaklystar D^s y$ weakly* as Radon measures (cf.\ also Theorem \ref{theo:SBV-compactness}). 
\end{enumerate} 
\end{rem}


It is straightforward to account for appropriate body forces in these functionals. Suppose $f^h:\Omega\to\R^2$ is a body force such that $h^{-2}f^h \to f$ in $L^1(\Omega;\R^2)$ and define the energy functionals under the load $f^h$ and the limiting energy functional by
\begin{align*}
  J^h(y) 
  & := \begin{cases} 
        I^h(y) - h^{-2} \int_{\Omega} y(x) \cdot f^h(x) \, dx & \mbox{for } y \in {\cal A}^h, \\ 
       + \infty & \mbox{otherwise,}
    \end{cases} 
\end{align*}
respectively, 
\begin{align*}
  J^0(y) 
  & := \begin{cases} 
        I^0(y) - \int_0^L \bar{y}(x_1) \cdot \bar{f}(x_1) \, dx & \mbox{for } y \in {\cal A},~ y(x) = \bar{y}(x_1) \mbox{ a.e., } \\ 
       + \infty & \mbox{otherwise,}
    \end{cases} 
\end{align*}
where $\bar{f} = \int_{-1/2}^{1/2} f(\cdot,s) \, ds$.

\begin{corollary}[Gamma-convergence and compactness]\label{cor:Gamma-compactness-forces} 
The functionals $J^h$ $\Gamma$-converge to $J^0$ on $SBV(\Omega; \R^2)$ with respect to the strong $L^1$-topology as $h \to 0$. 

If a sequence $y^h$ satisfies $\sup_h J^h(y^h) < \infty$, then there exists a subsequence (not relabeled) and a limiting $y \in {\cal A}$ such that 
$$ y^{h} \to y \quad \mbox{and} \quad 
   \nabla y^h \to \nabla y $$ 
strongly in $L^p$ for any $1 \le p < \infty$. 
\end{corollary}

Clamped boundary conditions can be included in this analysis conveniently by considering the enlarged domain $(-\eta, L + \eta)$ for an arbitrary $\eta > 0$ and, for given $y_0, e_0, y_L, e_L \in \R^2$ with $|y_0|, |y_L| < M$ and $|e_0| = |e_L| = 1$, defining the energy functionals 
\begin{align*}
  J^h_{\rm bv}(y) 
   := \begin{cases} 
      J^h(y)& \mbox{for } y \in {\cal A}_{\rm bv}^h, \\ 
       + \infty & \mbox{otherwise}
    \end{cases} 
\qquad\mbox{and}\qquad 
  J^0_{\rm bv}(y)  
   := \begin{cases} 
      J^0(y) & \mbox{for } y \in {\cal A}_{\rm bv}, \\ 
       + \infty & \mbox{otherwise,}
    \end{cases}
\end{align*}
where $J^h, {\cal A}^h, J^0, {\cal A}$ are the functionals introduced above, but on the domain $(-\eta, L + \eta)$, and 
\begin{align*}
  {\cal A}_{\rm bv}^h 
  = \big\{ y \in {\cal A}^h : 
    y(x) &= y_0 + x_1 e_0 + h x_2 e_0^{\perp} ~\mbox{ for }~ - \eta < x_1 < 0 ~\mbox{ and }~ \\ 
    y(x) &= y_L + (x_1 - L) e_L + h x_2 e_L^{\perp} ~\mbox{ for }~ L < x_1 < L + \eta \big\}, \\ 
  {\cal A}_{\rm bv} 
  = \big\{ y \in {\cal A} : 
    \bar{y}(x_1) &= y_0 + x_1 e_0 ~\mbox{ for }~ - \eta < x_1 < 0 ~\mbox{ and }~ \\ 
    \bar{y}(x_1) &= y_L + (x_1 - L) e_L ~\mbox{ for }~ L < x_1 < L + \eta \big\}. 
\end{align*}
Observe that finite energy deformations are rigid on $(-\eta, L + \eta) \setminus (0, L)$ so that this part does not contribute to the energy. In particular, $J^h_{\rm bv}(y)$ and $J^0_{\rm bv}(y)$ are in fact independent of the choice of $\eta$. However, these functionals account for the physically relevant possibility that cracks concentrate near the boundary $\{0, L\} \times (-\frac{1}{2}, \frac{1}{2})$ and a limiting deformation $y \in {\cal A}_{\rm bv}$ does not attain the boundary values 
$$ y(0) = y_0,~ y'(0) = e_0,~ y(L) = y_L,~ y'(L) = e_L $$ 
in the sense of traces. In such a case, non-attainment of the prescribed values at the left or right end of the beam is penalized in the limiting energy by the crack energy amount $\beta$ each. 

\begin{corollary}[Gamma-convergence and compactness]\label{cor:Gamma-compactness-forces-bv} 
Let $\eta > 0$. The functionals $J^h_{\rm bv}$ $\Gamma$-converge to $J^0_{\rm bv}$ on $SBV \big( (-\eta, L + \eta) \times (-\frac{1}{2}, \frac{1}{2}); \R^2 \big)$ with respect to the strong $L^1$-topology as $h \to 0$. 

If a sequence $y^h$ satisfies $\sup_h J^h_{\rm bv}(y^h) < \infty$, then there exists a subsequence (not relabeled) and a limiting $y \in {\cal A_{\rm bv}}$ such that 
$$ y^{h} \to y \quad \mbox{and} \quad 
   \nabla y^h \to \nabla y $$ 
strongly in $L^p$ for any $1 \le p < \infty$. 
\end{corollary} 

\begin{rem}
An analogous statement is true if clamped boundary conditions are imposed only at one end of the beam. 
\end{rem}

We will prove Corollaries \ref{cor:Gamma-compactness-forces} and \ref{cor:Gamma-compactness-forces-bv} at the end of Section \ref{sec:Gamma}. As a direct consequence of these results we obtain the convergence of almost minimizers to minimizers of the corresponding limiting functional. 
\begin{corollary}[Convergence of almost minimizers]\label{cor:conv-almost-min} If $(y^{h})$ is a sequence of almost minimizers of $J^h$ or of $J^h_{\rm bv}$, i.e.,
$$ J^h(y^{h}) - \inf J^h \to 0, 
   \quad\mbox{respectively,}\quad 
   J^h_{\rm bv}(y^{h}) - \inf J^h_{\rm bv} \to 0, $$
then there exists $y \in {\cal A}$, respectively, $y \in {\cal A}_{\rm bv}$, such that $y^{h} \to y$ in $L^1$ (for a subsequence) and $y$ minimizes $J^0$, respectively, $J^0_{\rm bv}$. 
\end{corollary}

\begin{proof} 
Immediate form Corollaries \ref{cor:Gamma-compactness-forces} and \ref{cor:Gamma-compactness-forces-bv}. 
\end{proof}

\section{Preliminaries: $\bm{SBV}$ and quantitative piecewise rigidity}\label{sec:Preliminaries} 

For convenience of the reader we first briefly review the definition of the space $SBV$ of special functions of bounded variation which serves as our basic model for deformations exhibiting both elastic regions and cracks. We then state the fundamental compactness theorem and a closure result in $SBV$. We also introduce the notion of Caccioppoli partitions. For an exhaustive treatment of $BV$ and $SBV$ functions we refer to \cite{Ambrosio-Fusco-Pallara:2000}. With these preparations we can finally state a quantitative piecewise geometric rigidity result in $SBV$ that was recently obtained in \cite{FriedrichSchmidt:15b} and which is a main ingredient in our proofs of Theorems \ref{theo:compactness} and \ref{theo:Gamma-convergence}. 

Suppose $\Omega \subset \R^d$ is a bounded Lipschitz domain. $y \in L^1(\Omega; \R^m)$ is said to be an element of $BV(\Omega; \R^m)$ if its distributional derivative $D y$ is a finite $\R^{m \times d}$-valued Radon measure. Accordingly, $D y$ can be decomposed into an absolutely continuous part $\nabla y$ with respect to the Lebesgue measure ${\cal L}^d$ and a singular part $D^s y$. If the Cantor part of $D^s y$ vanishes we say that $y$ is a special function of bounded variation and write $y \in SBV(\Omega, \R^m)$. In this case $D y$ takes the form 
\begin{align*}
  Dy
  = \nabla y {\cal L}^d + (y^+ - y^-) \otimes \nu_y {\cal H}^{d-1}\lfloor J_y.
\end{align*}
Here ${\cal H}^{d-1}$ denotes the $(d-1)$-dimensional Hausdorff measure, $J_y$ is an ${\cal H}^{d-1}$-rectifiable subset of $\Omega$, $\nu_y$ is the normal to $J_y$ and $y^+, y^-$ are the one-sided limits of $y$ at $J_y$. (If $y$ is a deformation, then $J_y$ is the `crack set' and $(y^+ - y^-) \otimes \nu_y$ measures the `crack opening'.) The subset of those $y \in SBV(\Omega; \R^m)$ for which $\nabla y \in L^p$ and ${\cal H}^{d-1}(J_y) < \infty$ is denoted $SBV^p(\Omega, \R^m)$. 

In the one-dimensional case $d = 1$ where $\Omega = (a,b)$ is an interval, the space $SBV^p((a,b), \R^m)$ coincides with the space $\PW^{1,p}((a,b); \R^m)$ of piecewise $W^{1,p}$ Sobolev functions consisting of those $y \in L^1((a,b); \R^m)$ for which 
\begin{align*}
  \exists \, a = t_0 < t_1 < \ldots < t_n = b  
  ~:~ y \in W^{1,p}((t_{i-1}, t_i); \R^m) \quad \forall \, i = 1, \ldots, n.  
\end{align*}
If $\{t_0, \ldots, t_{n}\}$ is the minimal set with this property, then $J_y = \{t_1, \ldots, t_{n-1}\}$ (and $\nu_y \equiv 1$). We may then assume that $y$ is uniformly continuous on the intervals $(t_{i-1}, t_i)$ and $y^{\mp}(t_i)$ are the limits of $y(t)$ as $t \to t_i$ from the left, respectively, from the right. For $d = 1$ we more generally also consider the spaces 
\begin{align*}
 \PW^{k,p}((a,b); \R^m) 
   &= \big\{ y \in L^1((a,b); \R^m) : \exists \, a = t_0 < t_1 < \ldots < t_n = b \\ 
   &\qquad\qquad \mbox{ such that } y \in W^{k,p}((t_{i-1}, t_i); \R^m) ~ \forall \, i = 1, \ldots, n \big\} 
\end{align*}
of piecewise $W^{k,p}$ regular functions, $k \in \N$. For a function $y$ in this space the minimal set $\{t_1, \ldots, t_{n-1}\}$ with $y \in W^{k,p}((t_{i-1}, t_i); \R^m)$ for $i = 1, \ldots, n$ (and $t_0 = a$, $t_n = b$) is $J_{y} \cup J_{y'} \cup \ldots \cup J_{y^{(k-1)}}$. 

We state the fundamental compactness result in $SBV$, first proved in \cite{Ambrosio:1989}, cf.\ also \cite[Theorems 4.7 and 4.8]{Ambrosio-Fusco-Pallara:2000}, as follows.  
\begin{theorem}\label{theo:SBV-compactness}
Let $(y_k)$ be a sequence in $SBV^p(\Omega; \R^m)$, $1 < p < \infty$, such that 
\begin{align*}
  \int_{\Omega} |\nabla y_k|^p \, dx
  + {\cal H}^{d-1}(J_{y_k}) + \| y_k \|_{L^{\infty}} 
  \le C 
\end{align*}
for some constant $C$ not depending on $k$. Then there exists a subsequence (not relabeled) and a function $y \in SBV^p(\Omega; \R^m)$ such that 
\begin{itemize} 
\item[(i)] $y_{k} \to y$ in $L^p(\Omega; \R^m)$, 
\item[(ii)] $\nabla y_k \weakly \nabla y$ in $L^p(\Omega; \R^{d \times m})$ and
\item[(iii)] $\liminf_{k \to \infty} {\cal H}^{d-1}(J_{y_k}) \ge {\cal H}^{d-1}(J_y)$.
\end{itemize} 
Moreover, $D^s y_k \weaklystar D^s y$ weakly* as Radon measures. 
\end{theorem}

In fact, we will need to apply this result only in the by far more elementary one-dimensional case. 
\begin{theorem}\label{theo:PW-compactness}
Let $(y_k)$ be a sequence in $\PW^{1,p}((0,L); \R^m)$, $1 < p < \infty$, such that 
\begin{align*}
  \int_0^L |y_k'|^p \, dt
  + \# J_{y_k} + \| y_k \|_{L^{\infty}} 
  \le C 
\end{align*}
for some constant $C$ not depending on $k$. Then there exists a subsequence (not relabeled), a function $\PW^{1,p}((0,L); \R^m)$ and a finite set $J \subset (0, L)$ such that 
\begin{itemize} 
\item[(i)] $y_{k} \to y$ in $L^p((0,L); \R^m)$, 
\item[(ii)] $y_k' \weakly y'$ in $L^p((0,L), \R^m)$ and
\item[(iii)] $J_{y_k} \to J \supset J_y$. 
\end{itemize} 
(In particular, $\liminf_{k \to \infty} \# J_{y_k} \ge \# J_y$.) 
\end{theorem}

In our proof of Theorem \ref{theo:Gamma-convergence} we will encounter deformations $y$ where a priori only the absolutely continuous part ${\cal E} y$ of the symmetrized derivative $\frac{1}{2}((Dy)^T + Dy)$ is controlled, which for $y \in SBV(\Omega; \R^d)$ is given by ${\cal E} y = \frac{1}{2}((\nabla y)^T + \nabla y)$. There are compactness results analogous to Theorem \ref{theo:SBV-compactness} in the more general function spaces $SBD$, see \cite{Bellettini-Coscia-DalMaso:1998}, and $GSBD$, see \cite{DalMaso:13}, that are more adapted to this situation. As in fact in our proof of Theorem \ref{theo:Gamma-convergence} we will only need a closure result, we will formulate the following direct consequence of \cite[Theorem 11.3]{DalMaso:13} for SBV functions only, thus circumventing the need of introducing the space $GSBD$ here. 

\begin{theorem} \label{theo:SBD-closure}
Let $(y_k)$ be a sequence in $SBV(\Omega; \R^d)$ such that for a continuous function $\psi : (0, \infty) \to (0, \infty)$ with $\lim_{s \to \infty} \frac{\psi(s)}{s} = \infty$ 
\begin{align*}
  \int_{\Omega} \psi(|y_k|) + \psi(|{\cal E} y_k|) \, dx
  + {\cal H}^{d-1}(J_{y_k})
  \le C 
\end{align*}
for some constant $C$ not depending on $k$. If $y_{k} \weakly y$ in $L^1(\Omega; \R^d)$ with $y \in SBV(\Omega; \R^d))$, then 
\begin{itemize} 
\item[(i)] $y_{k} \to y$ in $L^1(\Omega; \R^d)$, 
\item[(ii)] ${\cal E} y_k \weakly {\cal E} y$ in $L^1(\Omega; \R^{d \times d}_{\rm sym})$ and
\item[(iii)] $\liminf_{k \to \infty} {\cal H}^{d-1}(J_{y_k}) \ge {\cal H}^{d-1}(J_y)$.
\end{itemize} 
\end{theorem}

\begin{rem}\label{rem:DPdlVP}
By the Theorems of Dunford-Pettis and de la Vall{\'e}e-Poussin, the assumptions of Theorem \ref{theo:SBD-closure} are satisfied for $y, y_1, y_2, \ldots \in SBV(\Omega; \R^d)$ if $y_{k} \weakly y$ in $L^1(\Omega; \R^d)$, $({\cal E} y_k)_k$ is relatively weakly compact  in $L^1(\Omega; \R^{d \times d})$ and ${\cal H}^{d-1}(J_{y_k}) \le C$. 
\end{rem}

We say that a subset $E \subset \Omega$ has finite perimeter in $\Omega$ if the characteristic function $\chi_E$ belongs to $SBV(\Omega)$. In this case the jump set $J_{\chi_E}$ is denoted by $\partial^{\ast} E$ and called the reduced boundary of $E$. Then $\nabla \chi_E = 0$ a.e.\ and $D \chi_E$ is concentrated on $\partial^{\ast} E$ with $|D \chi_E| = {\cal H}^{d-1} \lfloor \partial^{\ast} E$. The perimeter of $E$ in $\Omega$ is defined as ${\rm Per}(E, \Omega) = {\cal H}^{d-1}(\partial^{\ast} E)$. 

A partition $(E_j)_j$ of $\Omega$ consisting of at most countably many sets $E_j$ of finite perimeter is called a Caccioppoli partition of $\Omega$ if $\sum_j {\rm Per}(E_j, \Omega) < \infty$.

Analogously to the purely elastic case treated in \cite{FJM:02}, the main ingredient into the derivation of an effective dimensionally reduced theory is the following quantitative piecewise geometric rigidity result which for an $SBV$ function $y$ estimates the deviation of $\nabla y$ from a piecewise constant $\SO(2)$-valued mapping with controlled jump part in terms of its energy. This result, which also uses a novel Korn-Poincar\'e inequality in $SBD$ obtained in \cite{Friedrich:15a}, was recently proved in \cite{FriedrichSchmidt:15b}. It provides a quantitative version of a Liouville type piecewise rigidity result by Chambolle, Giacomini and Ponsiglione in \cite{Chambolle-Giacomini-Ponsiglione:2007}. We state it here in a form which directly follows from (the proof of) \cite[Theorem 2.1 and Remark 2.2]{FriedrichSchmidt:15b}. 

Suppose that $W$ satisfies the assumptions of (i) regularity, (ii) frame indifference and (iii) non-degeneracy from Section \ref{sec:Intro}. Also set 
$$ SBV_M(\Omega, \R^2) 
   = \big\{ y \in SBV(\Omega, \R^2) : \| \nabla y \|_{L^{\infty}} \le M,~ {\cal H}^1(J_y) < \infty \big\}. $$ 
for $M > 0$ and $\Omega \subset \R^2$ open. 

\begin{theorem}\label{theo:quant-rig}
Let $\Omega \subset \R^2$  open, bounded with Lipschitz boundary. Let $M>0$ and $0 < \eta, \rho < 1$. Then there are constants $C=C(\Omega,M,\eta)$, $\hat{C}=\hat{C}(\Omega,M,\eta,\rho)$ and universal $c, \bar{c} > 0$ such that the following holds for $\eps >0$ small enough. 

If $y \in SBV_M(\Omega; \R^2) \cap L^2(\Omega; \R^2)$ is such that 
$$ {\cal H}^{1}(J_y) \le M 
   \quad\mbox{and}\quad 
   \int_\Omega \dist^2(\nabla y,\SO(2)) \le M\eps $$ 
and we set $\tilde{\eps} = \int_\Omega \dist^2(\nabla y,\SO(2) ) + \eps {\cal H}^1(J_y)$ and $\Omega_\rho = \{ x \in \Omega : \dist(x, \partial \Omega) > \bar{c} \rho\}$, then there is an open set $\Omega_y$ with $|\Omega_{\rho} \setminus\Omega_y| \le C \rho \eps^{-1} \tilde{\eps}$, a modification  $\hat{y} \in SBV_{cM}(\Omega) \cap L^2(\Omega; \R^2)$ with $\| \hat{y} - y \|^2_{L^2(\Omega_y)} +  \| \nabla \hat{y} - \nabla y \|^2_{L^2(\Omega_y)}\le C \rho \tilde{\eps}$ satisfying the estimates
\begin{align}\label{rig-eq: crack le}
  {\cal H}^1 ( J_{\hat{y}} \cap \Omega_{\rho} ) \le  C \eps^{-1} \tilde{\eps} 
\end{align}
and 
\begin{align}\label{rig-eq: energy le}
  \frac{1}{\eps} \int_{\Omega_\rho} W(\nabla \hat{y}) \, dx 
  \le \frac{1}{\eps}\int_{\Omega} W(\nabla y) \, dx + {\cal H}^1(J_y) + C \rho \eps^{-1} \tilde{\eps}
\end{align}
with the following properties: We find  a Caccioppoli partition ${\cal P} = (P_j)_j$ of $\Omega_\rho$ with 
$$ \sum_j \frac{1}{2} \mathrm{Per}(P_j,\Omega_{\rho}) \le {\cal H}^1(J_y) 
   + C \rho \eps^{-1} \tilde{\eps} $$ 
and, for each $P_j$, a corresponding rigid motion $R_j \cdot + c_j$, $R_j \in \SO(2)$ and $c_j \in \R^2$, such that the function $\hat{u}: \Omega \to \R^2$ defined by
\begin{align}\label{rig-eq: u def2} 
\hat{u}(x) := \begin{cases} \hat{y}(x) - (R_j\,x +c_j) & \ \ \text{ for } x \in P_j \\
                      0                      & \ \ \text{ for } x \in \Omega \setminus \Omega_\rho \end{cases}
\end{align}
satisfies the estimates
\begin{align}\label{rig-eq: main properties2}
\begin{split}
(i) & \ \, {\cal H}^{1}(J_{\hat{u}}) \le C \eps^{-1} \tilde{\eps}, \ \  \ \ \  \ \ \  \ \ \ \   \ \ \   \ \  \ \ \  \ \, 
(ii) \,  \ \| \hat{u}\|^2_{L^2(\Omega_\rho)} \le \hat{C} \tilde{\eps}, \\
(iii) & \ \, \sum\nolimits_j \| e(R^T_j \nabla \hat{u})\|^2_{L^2(P_j)} \le \hat{C} \tilde{\eps},  \  \ \ \  \  \ \ \,  
(iv)  \ \, \| \nabla \hat{u}\|^2_{L^2(\Omega_\rho)} \le \hat{C} \tilde{\eps}^{1-\eta}, 
\end{split}
\end{align}
where $e(G) = \frac{G + G^T}{2}$ for all $G \in \R^{2 \times 2}$. Moreover, if $y \in L^{\infty}(\Omega)$, then $\| \hat{y} \|_{L^{\infty}} \le c \| y \|_{L^{\infty}}$. 
\end{theorem}

\section{Modification, approximation and compactness}\label{sec:compactness}

Throughout this section we assume that $y^h \in {\cal A}^h$ is a bounded energy sequence of deformations verifying 
\begin{align}\label{eq:Ihyh-bound}
  I^h(y^h) 
  \le C 
\end{align}
for a constant $C$ independent of $h$. As it will be convenient in the sequel, we also introduce the rescaled deformations $w^h(x) = h^{-1} v^h(h x) = h^{-1} y^h(h x_1, x_2)$ defined on $h^{-1} \Omega_h$ which are elements of $SBV(h^{-1} \Omega_h; \R^2)$ with $\| w \|_{L^{\infty}} \le M h^{-1}$ and $\| \nabla w \|_{L^{\infty}} \le M$. For $w^h$ the energy bound \eqref{eq:Ihyh-bound} implies 
\begin{align}\label{eq:energy-bound}
  \int_{h^{-1}\Omega_h} \dist^2(\nabla w^h, \SO(2)) \, dx + h {\cal H}^1(J_{w^h}) 
  \le Ch. 
\end{align}
We will first introduce local modifications of $w^h$ and construct an approximating almost $\SO(2)$-valued mapping $\tilde{R}^h$. Next we will determine the limiting behavior of $\tilde{R}^h$ and its dependence on the modification parameters. Finally, the proof of Theorem \ref{theo:compactness} on compactness is given.

\subsection{Approximation and modification}

We fix $n \in \N$ and cover $(0, \lfloor \frac{L-h}{h} \rfloor) \times (-\tfrac{1}{2}, \tfrac{1}{2})$ with the rectangles 
$$ Q_a = (a-1-\tfrac{1}{2n}, a+\tfrac{1}{2n}) \times (-\tfrac{1}{2}, \tfrac{1}{2}), \quad a = 1, \ldots, N = \lfloor \tfrac{L-h}{h} \rfloor. $$  
Let $\rho > 0$ be a small parameter. Then choose $h_0 = h_0(\rho)$ such that Theorem \ref{theo:quant-rig} applies to sets of the form $I \times (-\tfrac{1}{2}, \tfrac{1}{2})$, where $I$ is an interval of length $1+\frac{1}{2n}$, $1+\frac{1}{n}$, $2+\frac{1}{2n}$ or $2+ \frac{1}{n}$, with $\eta = \frac{1}{10}$, $\rho$ as given and all $\eps = h \le h_0$. We also fix a threshold value $\frac{1}{2} < \lambda < 1$. Eventually $\lambda$ will be sent to $1$ after $\rho$ has been sent to $0$. 

In the following estimates we will only consider $h$ with $h < h_0$; generic constants that are independent of $\rho$ and $\lambda$ will be denoted by $C$ whereas constants that may depend on $\rho$ or $\lambda$ are called $\hat{C}$. $n$ will be fixed until the very end of the proof of Theorem \ref{theo:Gamma-convergence}(i) and both $C$ and $\hat{C}$ may depend on $n$. 

We define a set ${\cal G} = {\cal G}(h, \lambda, \rho)$ of `good rectangles' and the complementary set ${\cal B} = {\cal B}(h, \lambda, \rho)$ of `bad rectangles' by 
\begin{align*}
  {\cal G} 
  &= \Big\{ Q_a : {\cal H}^1 \big( J_{w^h} \cap (Q_{a-1} \cup Q_{a} \cup Q_{a+1}) \big) < \lambda \\ 
  &\qquad\qquad\qquad \mbox{ and }~ \int_{Q_a} \dist^2 \big( \nabla w^h, \SO(2) \big) \le \frac{1}{2} (M-1) h_0 \Big\}, \\ 
  {\cal B} 
  &= \{Q_1, \ldots, Q_N \} \setminus {\cal G}. 
\end{align*}
Considering the restriction of $w^h$ to $Q_a$ we set 
\begin{align*} 
  \eps_a = \int_{Q_a} \dist^2(\nabla w^h, \SO(2)) + h {\cal H}^1 ( J_{w^h} \cap Q_a ). 
\end{align*}
By the energy bound \eqref{eq:energy-bound} we then have 
\begin{align}\label{eq:sum-epsa-bound} 
  \sum_{a = 1}^N \eps_a 
  \le 2 \int_{h^{-1}\Omega_h} \dist^2 \big( \nabla w^h, \SO(2) \big) \, dx + 2 h {\cal H}^1(J_{w^h}) 
  \le Ch. 
\end{align}

We begin with the following elementary observation on the energy and number of `bad squares'. 
\begin{lemma}\label{lemma:B-bounds} 
For every $h \le \min\{ \frac{1}{2} (M-1) c, 1 \} h_0$ with $c$ from \eqref{eq:W-lower-bound} we have 
\begin{itemize} 
\item[(i)] $\int_{Q_a} W \big( \nabla w^h \big) + h {\cal H}^1 \big( J_{w^h} \cap (Q_{a-1} \cup Q_{a} \cup Q_{a+1}) \big) \ge \lambda h$ if $Q_a \in {\cal B}$ and 
\item[(ii)] $\# {\cal B} \le C$ for a constant $C$ which is independent of $h$ and $\rho$.
\end{itemize} 
\end{lemma} 

\begin{proof}
(i) immediately follows as for $Q_a \in {\cal B}$ 
\begin{align*}
  \int_{Q_a} W(\nabla w^h) + h {\cal H}^1 \big( J_{w^h} \cap (Q_{a-1} \cup Q_{a} \cup Q_{a+1}) \big) 
  \ge \min\{ \tfrac{1}{2} (M-1) c h_0, h \lambda \} 
  = \lambda h. 
\end{align*}
Summing over $Q_a \in {\cal B}$ and recalling the energy bound \eqref{eq:Ihyh-bound} we then get 
\begin{align*} 
  \#{\cal B} \lambda h
  \le 5 h I^h(y^h) 
  \le Ch 
\end{align*}
from which (ii) follows. 
\end{proof}

\subsubsection*{Modification on good rectangles}

We now discuss a modification and approximation to $y^h$ on good rectangles which allows to estimate the local variations of $y^h$ in a sufficiently sharp way. We begin by considering a single rectangle $Q_a \in {\cal G}$. First note that, by construction, $\eps_a \le (M-1) h_0 + h \lambda \le M h_0$ and ${\cal H}^1(J_{w^h} \cap Q_a) \le \lambda \le M$.

Applying Theorem \ref{theo:quant-rig} with $\eps = h$ and $\tilde{\eps} = \eps_a$ to $w^h$ on $Q_a$ we obtain an open set $V_a \subset Q_a$ and a modification $\hat{w}_a \in SBV_{cM}(Q_a) \cap L^2(Q_a)$ with $\| \hat{w}_a \|_{L^{\infty}(Q_a)} \le c \| w^h \|_{L^{\infty}(Q_a)}$ such that for $Q_{a, \rho} = \{ x \in Q_a : \dist(x, \partial Q_a) < \bar{c} \rho \}$ the modification error is estimated by 
\begin{align}\label{eq:Va-est}
   h |Q_{a,\rho} \setminus V_a| + \| \hat{w}_a - w^h \|_{L^2(V_a)}^2 + \| \nabla \hat{w}_a - \nabla w^h \|_{L^2(V_a)}^2 \le C \rho \eps_a 
\end{align}
and $\hat{w}_a$ satisfies the energy estimates 
\begin{align}\label{eq:wa-crack-est}
  {\cal H}^1 ( J_{\hat{w}_a} \cap Q_{a,\rho} ) \le  C h^{-1} \eps_a 
\end{align}
and 
\begin{align}\label{eq:wa-elast-est}
  \frac{1}{h} \int_{Q_{a,\rho}} W(\nabla \hat{w}_a) \, dx 
  \le \frac{1}{h}\int_{Q_a} W(\nabla w^h) \, dx + {\cal H}^1(J_{w^h} \cap Q_a) + C \rho h^{-1} \eps_a. 
\end{align}
Moreover, there exists a Caccioppoli partition $(P_{a,j})_j$ of $Q_{a, \rho}$ with 
\begin{align}\label{eq:Cacc-per-est}
  \sum_j \frac{1}{2}\mathrm{Per}(P_{a,j}, Q_{a,\rho}) 
  \le {\cal H}^1(J_{w^h} \cap Q_{a}) + C \rho h^{-1} \eps_a   
\end{align}
and for each $P_{a,j}$ an orthogonal matrix $R_{a,j} \in \SO(2)$ and a translation vector $c_{a,j} \in \R^2$ such that, setting 
\begin{align*}
  \hat{u}_a(x) 
  := \begin{cases} 
        \hat{w}_a(x) - (R_{a,j}\,x + c_{a,j}) & \mbox{for } x \in P_{a,j}, \\
                      0                      & \mbox{for } x \in Q_a \setminus Q_{a,\rho}
     \end{cases}
\end{align*}
one has 
\begin{align}\label{eq:ua-crack-est}
  {\cal H}^{1}(J_{\hat{u}_a}) \le C h^{-1} \eps_a \quad 
\end{align}
and 
\begin{align}\label{eq:ua-estimates}
  \| \hat{u}_a \|^2_{L^2(Q_{a,\rho})} \le \hat{C} \eps_a, \quad 
  \sum\nolimits_j \|e(R^T_{a,j} \nabla \hat{u}_a) \|^2_{L^2(P_{a,j})} \le \hat{C}\eps_a, \quad 
  \| \nabla \hat{u}_a \|^2_{L^2(Q_{a,\rho})} \le \hat{C}\eps_a^{\frac{9}{10}} 
\end{align}
for a constant $C$ which is independent of $\rho$ and a constant $\hat{C}=\hat{C}(\rho)$. Note that since $\| \hat{w}_a \|_{L^{\infty}} \le c \| w^h \|_{L^{\infty}} \le C h^{-1}$ it is not restrictive to assume that 
\begin{align}\label{eq:caj-bound}
  | c_{a,j} | \le C h^{-1} \quad \forall \, a, j. 
\end{align}

In the following we will assume that the numbering of the partition $(P_{a,j})_j$ is such that the area $|P_{a,1}|$ is maximal. Note that for any $\mu < 1 - 2 \bar{c} \rho$, by the isoperimetric inequality, under the constraint $\mathrm{Per}(P_{a,1}, Q_{a,\rho}) \le \mu$ the value $|Q_{a,\rho} \setminus P_{a,1}|$ is at most the area of the intersection of $Q_{a,\rho}$ with a disc of radius $\frac{2\mu}{\pi}$ centered at one of its corners, i.e., $\frac{\mu^2}{\pi}$. So 
\begin{align}\label{eq:Per-Pa1-est}
  |P_{a,1}| 
  \ge |Q_{a,\rho}| - \frac{\mu^2}{\pi} 
  \ge \frac{1}{n} + \frac{2}{3}. 
\end{align}
In particular this estimate is satisfied for $Q_a \in {\cal G}$ with $\mu = \lambda + C \rho$ for sufficiently small $\rho$ by \eqref{eq:Cacc-per-est}.

\subsubsection*{Estimates on overlapping rectangles} 

Let $r_{a,1} : \R^2 \to \R^2$ be the affine mapping 
$$ r_{a,1}(x) = R_{a,1}x + c_{a,1}. $$ 
Our aim is to compare the mappings $r_{a,1}$ and the values of $R_{a,1}$ of their gradients on overlapping rectangles. Unfortunately, this difference cannot directly be estimated with the help of $\nabla \hat{w}_a$ as we only control the symmetric part of $R^T_{a,j} \nabla u_a$ whereas the difference of two rotations $R_{a,1}$ and $R_{b,1}$ to leading order is measured by the skew symmetric part of $R_{a,1}^TR_{b,1}$. Nevertheless, we have the following estimate. 
\begin{lemma}\label{lemma:RaRb-est}
If $Q_a, Q_b \in {\cal G}$, $b = a+1$, are overlapping rectangles, then 
$$ \| r_{a,1} - r_{b,1} \|_{L^{\infty}(Q_a \cup Q_b)}^2 
   + | R_{a,1} - R_{b,1} |^2 
   \le \hat{C} ( \eps_a + \eps_b ). $$ 
\end{lemma}

\begin{proof} 
On $Q_{a,\rho}$ we have 
\begin{align*}
  \| \hat{w}_a - r_{a,1} \|_{L^2(P_{a,1})} 
  = \| \hat{u}_a \|_{L^2(_{a,1})} 
  \le \hat{C} \sqrt{\eps_a}. 
\end{align*}
by \eqref{eq:ua-estimates}. For the original deformation $w^h$ we thus obtain 
\begin{align}\label{eq:w-Ra1} 
  \| w^h - r_{a,1} \|_{L^2(P_{a,1} \cap V_a)} 
  \le \hat{C} \sqrt{\eps_a} 
\end{align} 
as well by using \eqref{eq:Va-est} with $|P_{a,1}| \ge \frac{1}{n} + \frac{2}{3}$ according to \eqref{eq:Per-Pa1-est}. 

Suppose that $Q_b = Q_{a+1} \in {\cal G}$ overlaps with $Q_a$ and set $Q_{(a,b)} = Q_a \cup Q_b$. The same analysis applied to $Q_{(a,b)}$ yields a set $V_{(a,b)} \subset Q_{(a,b),\rho} = \{ x \in Q_{(a,b)} : \dist(x, \partial  Q_{(a,b)}) \ge 1 - \bar{c} \rho\}$ with $|Q_{(a,b),\rho} \setminus V_{(a,b)}| \le C \rho h^{-1} (\eps_a + \eps_b)$, a set $P_{(a,b),1} \subset Q_{(a,b),\rho}$ with $|P_{(a,b),1}| \ge \frac{1}{n} + \frac{5}{3}$ and an affine mapping $r_{(a,b),1} = R_{(a,b),1} \cdot + c_{(a,b),1}$ with $R_{(a,b),1} \in \SO(2)$ such that 
\begin{align}\label{eq:w-Rab1} 
  \| w^h - r_{(a,b),1} \|_{L^2(P_{(a,b),1} \cap V_{(a,b)})} 
  \le \hat{C} \sqrt{\eps_a + \eps_b}. 
\end{align} 
Combining \eqref{eq:w-Ra1} and \eqref{eq:w-Rab1} we thus find 
\begin{align*}
  \| r_{a,1} - r_{(a,b),1} \|_{L^2(P_{a,1} \cap P_{(a,b),1} \cap V_a \cap V_{(a,b)})} 
  \le \hat{C} \sqrt{\eps_a + \eps_b}. 
\end{align*}
Since $|P_{a,1} \cap P_{(a,b),1} \cap V_a \cap V_{(a,b)}| \ge \frac{1}{3} - C \rho$, it is elementary to see that 
$$ \| r_{a,1} - r_{(a,b),1} \|_{L^{\infty}(Q_a \cup Q_b)} 
   + | R_{a,1} - R_{(a,b),1} | 
   \le \hat{C} \sqrt{\eps_a + \eps_b}. $$ 

In complete analogy we have $\| r_{b,1} - r_{(a,b),1} \|_{L^{\infty}(Q_a \cup Q_b)} + | R_{b,1} - R_{(a,b),1} | \le C \sqrt{\eps_a + \eps_b}$ and may thus conclude that 
$$ \| r_{a,1} - r_{b,1} \|_{L^{\infty}(Q_a \cup Q_b)} 
   + | R_{a,1} - R_{b,1} | \le \hat{C} \sqrt{\eps_a + \eps_b}. $$ 
\end{proof}

\subsection{Compactness of interpolated rotations} 

Our aim is now to interpolate the mappings $r_{a,1}$ and rotations $R_{a,1}$ to obtain function $\tilde{r}^h$ and $\tilde{R}^h$ on $(0,L)$ whose limiting behavior as $h \to 0$ can be analyzed. We remark at this point that mollification techniques are not appropriate as for smooth approximations in the presence of cracks we would loose sharp control over their derivatives. On the other hand, a simple piecewise affine, respectively, constant, interpolation would in principle be sufficient to prove the compactness statement in Theorem \ref{theo:compactness}. However, in computing the precise energy asymptotics in Section \ref{sec:Gamma} also such a strategy turns out to be insufficient as this procedure introduces artificial fracture contributions that cannot be estimated suitably.

\subsubsection*{Interpolation of ${\bm r}_{{\bm a}{\bm ,}{\bm 1}}$ and ${\bm R}_{{\bm a}{\bm ,}{\bm 1}}$}
 
Our aim is to interpolate $r_{a,1}$ and $R_{a,1}$ smoothly on intervals covered by good rectangles and extrapolate to functions on $(0,L)$ which only jump (at most once) on each component of the beam which is covered by bad rectangles. Let 
\begin{align}\label{eq:cla-B-def}
  I^{\cal B} &= \{ x_1 \in (0, L) : (h^{-1} x_1, 0) \in Q_{a} \mbox{ for some } Q_a \in {\cal B} \}
\end{align}
be the rescaled projection of the bad rectangles onto the $x_1$-axis. If $(p, q)$ is a connected component of $I^{\cal B}$ with $h^{-1} (p, q)$ covered, say, by $Q_a, \ldots, Q_b \in {\cal B}$, we set 
\begin{align*} 
      \tilde{r} (x_1) 
      &= h r_{a-1,1}(h^{-1}x_1,0) & \mbox{and} & & 
      \tilde{R} (x_1) 
      &= R_{a-1,1} & & \mbox{for } p \le x_1 < \tfrac{p+q}{2}, \\ 
      \tilde{r} (x_1) 
      &= h r_{b+1,1}(h^{-1}x_1,0) & \mbox{and} & & 
      \tilde{R} (x_1) 
      &= R_{b+1,1} & & \mbox{for } \tfrac{p+q}{2} < x_1 \le q 
   \end{align*}   
in cases $a \ne 1$, $b \ne N$. If $a = 1$ we set $\tilde{r}(x_1) = h r_{b+1,1}(h^{-1}x_1,0)$ and $\tilde{R}(x_1) = R_{b+1,1}$ for all $x_1 \in [0,q]$. In case $b = N$ we set $\tilde{r}(x_1) = h r_{a-1,1}(h^{-1}x_1,0)$ and $\tilde{R}(x_1) = R_{a-1,1}$, for all $x_1 \in [p,L]$. 

Next we set $\tilde{r}(x_1) = h r_{1,1}(h^{-1}x_1,0)$, $\tilde{R}(x_1) = R_{1,1}$ on $(0,\frac{h}{2})$ if $Q_1 \in {\cal G}$ and $\tilde{r}(x_1) = h r_{N,1}(h^{-1}x_1,0)$, $\tilde{R}(x_1) = R_{N,1}$ for $x \in ((N - \frac{1}{2})h, L)$ if $Q_{N} \in {\cal G}$. 

Now we choose a partition $(\varphi_a)_{a \in \N}$ of unity on $[0, \infty)$, i.e., $\sum_a \varphi_a \equiv 1$ on $(0, \infty)$, such that each $\varphi_a$ is smooth, supported on $(a - 1 - \frac{1}{2n} + \bar{c} \rho, a + \frac{1}{2n} - \bar{c} \rho)$ and satisfies $|\varphi_a'| \le 2n$ for $a > 1$. We then define $\tilde{r}$ and $\tilde{R}$ on the remaining part of $(0,L)$ by setting 
\begin{align}\label{eq:tR-def}
  \tilde{r}(x_1) 
  = \sum_{a : Q_a \in {\cal G}} \varphi_a(h^{-1} x_1) h r_{a,1}(h^{-1}x_1,0), \qquad 
  \tilde{R}(x_1) 
  = \sum_{a : Q_a \in {\cal G}} \varphi_a(h^{-1} x_1) R_{a,1} 
\end{align}
if $\tilde{r}(x_1)$ and $\tilde{R}(x_1)$ have not been defined previously. In this way we indeed obtain piecewise smooth functions $\tilde{r}\in \PW^{1,2}((0,L); \R^2)$ and $\tilde{R} \in \PW^{1,2}((0,L); \R^{2 \times 2})$ with $J_{\tilde{r}} \cup J_{\tilde{R}} \subset I^{\cal B}$ and $\# ( J_{\tilde{r}} \cup J_{\tilde{R}} ) \le \#{\cal B}$ as desired. 

We also note that for $x \in Q_{a,\rho}$, $Q_a \in {\cal G}$, 
\begin{align}\label{eq:tR-Ra1-full-est}
   |\tilde{R}^h(h x_1) - R_{a,1}|^2 \le \hat{C} h 
\end{align}
by construction of $\tilde{R}$ in \eqref{eq:tR-def} and the fact that for $Q_a, Q_{a+1} \in {\cal G}$ 
\begin{align}\label{eq:tR-Ra1-est}
  | \varphi_a R_{a,1} + \varphi_{a+1} R_{a+1,1} - R_{a,1} | 
  = | \varphi_{a+1} | | R_{a+1,1} - R_{a,1} | 
  \le \hat{C} \sqrt{\eps_a + \eps_{a+1}} 
  \le \hat{C} \sqrt{h} 
\end{align}
by Lemma \ref{lemma:RaRb-est} and \eqref{eq:sum-epsa-bound}. Similarly we have 
\begin{align}\label{eq:tr-ra1-full-est}
  |h^{-1}\tilde{r}^h(h x_1) - r_{a,1}(x)|^2 \le \hat{C} 
\end{align}
if $x \in Q_{a,\rho}$, $Q_a \in {\cal G}$. This follows from the construction of $\tilde{r}$ in \eqref{eq:tR-def} since 
\begin{align*}
  | \varphi_a r_{a,1} + \varphi_{a+1} r_{a+1,1} - r_{a,1} | 
  = | \varphi_{a+1} | | r_{a+1,1} - r_{a,1} | 
  \le \hat{C} \sqrt{h} 
\end{align*}
whenever $Q_a, Q_{a+1} \in {\cal G}$ by Lemma \ref{lemma:RaRb-est} in analogy with \eqref{eq:tR-Ra1-est} and trivially $|r_{a,1}(x) - r_{a,1}(x_1,0)| \le C$.

\begin{lemma}\label{lemma:tR-PW}
There are constants $C, \hat{C} > 0$ such that 
\begin{align*}
  (i)& \ \      \|   \dist(\tilde{R}, \SO(2)) \|_{L^{\infty}(0,L)} \le \hat{C} \sqrt{h}, \ \ \ \ \ \ \ 
  (ii) \ \      \| \tilde{r} \|_{L^{\infty}(0,L)} \le C, \\ 
  (iii)& \ \    \| \tilde{R}' \|_{L^2(0,L)} \le \hat{C}, \ \ \ \ \ \ \ \ \ \ \ \ \ \ \ \ \ \ \ \ \ \ \ \ \ \ \ 
  (iv) \ \      \| \tilde{r}' - \tilde{R} \e_1 \|_{L^2(0,L)} \le \hat{C} h^2. 
\end{align*}
(In particular, $\| \tilde{r}' \|_{L^2(0,L)} \le \hat{C}$.) 
\end{lemma}

\begin{proof}
On $(0,\frac{h}{2}) \cup ((N - \frac{1}{2})h, L) \cup I^{\cal B}$, the estimates on $\tilde{R}$ are trivially satisfied while the estimates on $\tilde{r}$ follow immediately from \eqref{eq:caj-bound} and $\tilde{r}'(x_1) = \tilde{R} \mathbf{e}_1$ for $x_1 \in (0,\frac{h}{2}) \cup ((N - \frac{1}{2})h, L) \cup I^{\cal B}$. Recalling that $\# {\cal B} \le C$ by Lemma \ref{lemma:B-bounds}, it therefore suffices to show that for every connected component $[p,q]$ of $(\frac{h}{2}, (N - \frac{1}{2})h) \setminus I^{\cal B}$ 
$$ \| \dist(\tilde{R}, \SO(2)) \|_{L^{\infty}} \le \hat{C} \sqrt{h}, \quad  
   \| \tilde{r} \|_{L^{\infty}} \le \hat{C}, \quad 
   \| \tilde{R}' \|_{L^2} \le \hat{C}
   \quad\mbox{and}\quad 
   \| \tilde{r}' \|_{L^2} \le \hat{C}. $$ 

The first two of these inequalities are immediate from \eqref{eq:tR-Ra1-full-est}, \eqref{eq:tr-ra1-full-est} and \eqref{eq:caj-bound}. In order to prove the other two inequalities we first choose $Q_a, Q_b \in {\cal G}$ such that $p \in Q_a$, $q \in Q_b$. Then we compute, again using Lemma \ref{lemma:RaRb-est} and \eqref{eq:sum-epsa-bound},   
\begin{align*}
  \int_p^q | \tilde{R}'(x_1) |^2 \, dx_1  
  &= \int_p^q \Big| \sum_{i=a}^{b} h^{-1} \varphi_i'(h^{-1} x_1) R_{i,1} \Big|^2 \, dx_1 \\ 
  &= h \int_{h^{-1}p}^{h^{-1}q} \Big| \sum_{i=a}^{b} h^{-1} \varphi_i'(x_1) R_{i,1} \Big|^2 \, dx_1 \\ 
  &= h^{-1} \sum_{i=a}^{b-1} \int_{i - \frac{1}{2n} + \bar{c} \rho}^{i + \frac{1}{2n} - \bar{c} \rho} \big| \varphi_i'(x_1) R_{i,1} + \varphi_{i+1}'(x_1) R_{i+1,1} \big|^2 \, dx_1 \\ 
  &= h^{-1} \sum_{i=a}^{b-1} \int_{i - \frac{1}{2n} + \bar{c} \rho}^{i + \frac{1}{2n} - \bar{c} \rho} \big| \varphi_i'(x_1) R_{i,1} - \varphi_i'(x_1) R_{i+1,1} \big|^2 \, dx_1 \\ 
  &\le \hat{C} h^{-1} \sum_{i=a}^{b-1} ( \eps_{i} + \eps_{i+1} ) ~
  \le \hat{C}, 
\end{align*}
where we have used that $\varphi_i' + \varphi_{i+1}' = 0$ on $(i - \frac{1}{2n}, i + \frac{1}{2n})$ and that $|\varphi_i'| \le 2n$. Finally, a completely analogous estimate shows that the $L^2$-Norm on $[p, q]$ of first sum in 
$$ \tilde{r}'(x_1) 
   = \sum_{i=a}^{b} h^{-1} \varphi_i'(h^{-1} x_1) h r_{i,1}(h^{-1} x_1, 0) + \sum_{i=a}^{b} \varphi_i(h^{-1} x_1) R_{i,1} \mathbf{e}_1 $$ 
is bounded by $\hat{C} h^2$ while the second one is equal to $\tilde{R} \mathbf{e}_1$. 
\end{proof}

\subsubsection*{Convergence of $\tilde{\bm r}$ and $\tilde{\bm R}$} 

So as to highlight their $h$-dependence we also write $\tilde{r} = \tilde{r}^h$ and $\tilde{R} = \tilde{R}^h$. 

\begin{lemma}\label{lemma:tilde-R-compact}
There are a subsequence (not relabeled) $h \to 0$, mappings $\bar{r} \in \PW^{1,2}((0,L); \R^2)$, $\bar{R} \in \PW^{1,2}((0,L); \SO(2))$ and a finite set $J \subset (0, L)$ such that for every (sufficiently small) $\rho > 0$, 
\begin{itemize}
\item[(i)] $\tilde{R}^h \to \bar{R}$ and  $\tilde{r}^h \to \bar{r}$ strongly in $L^2$, 
\item[(ii)] $(\tilde{R}^h)' \weakly \bar{R}'$ and $(\tilde{r}^h)' \weakly \bar{r}'$ weakly in $L^2$ and 
\item[(iii)] $J_{\tilde{R}^h} \cup J_{\tilde{r}^h} \to J \supset J_{\bar{R}} \cup J_{\bar{r}}$. 
\end{itemize}
\end{lemma} 

So in particular, the limiting $\R^2$-, respectively, $\SO(2)$-valued, functions $\bar{r}, \bar{R}$ do not depend on $\rho$. 

\begin{proof} 
For fixed $\rho = \rho_0 > 0$ (small) this readily follows from Lemmas \ref{lemma:tR-PW} and Theorem \ref{theo:PW-compactness} as there are only a bounded number of components of $I^{\cal B}$ and by construction $\tilde{R}^h$ and $\tilde{r}^h$ jump at most once in every such component while they are smooth outside $I^{\cal B}$. 

Now consider an arbitrary $0 < \rho \le \rho_0$. For each $Q_a \in {\cal G}(\rho)$ we then write \eqref{eq:w-Ra1} with arguments $\rho$ to highlight the dependence on $\rho$ in the form 
\begin{align*} 
  \| w^h - r_{a,1}(\rho) \|_{L^2(P_{a,1}(\rho) \cap V_a(\rho))} 
  \le \hat{C}(\rho) \sqrt{\eps_a}.  
\end{align*} 
Setting $W(\rho,\rho_0) = P_{a,1}(\rho) \cap P_{a,1}(\rho_0) \cap V_a(\rho) \cap V_a(\rho_0)$ we see that 
\begin{align*} 
  \| r_{a,1}(\rho) - r_{a,1}(\rho_0) \|_{L^2(W(\rho,\rho_0))} 
  \le \big( \hat{C}(\rho) + \hat{C}(\rho_0) \big) \sqrt{\eps_a}, 
\end{align*} 
where $|W(\rho,\rho_0)| \ge \frac{1}{3} - C \rho$ by \eqref{eq:Per-Pa1-est}. But then also 
\begin{align*} 
 \| r_{a,1}(\rho) - r_{a,1}(\rho_0) \|_{L^{\infty}(Q_a)} + | R_{a,1}(\rho) - R_{a,1}(\rho_0) | 
   \le \big( \hat{C}(\rho) + \hat{C}(\rho_0) \big) \sqrt{\eps_a}. 
\end{align*} 
Combining this estimate with \eqref{eq:tr-ra1-full-est} and \eqref{eq:tR-Ra1-full-est} we obtain 
\begin{align}\label{eq:rRrhoeho0} 
\begin{split}
  &h^{-1} \| \tilde{r}^h(\rho)(h x_1) - \tilde{r}^h(\rho_0)(h x_1) \|_{L^{\infty}(Q_a)} + | \tilde{R}^h(\rho)(h x_1) - \tilde{R}^h(\rho_0)(h x_1) | \\ 
  &~~ \le \big( \hat{C}(\rho) + \hat{C}(\rho_0) \big) \sqrt{h}. 
\end{split}
\end{align} 

Now recalling that by Lemma \ref{lemma:B-bounds} both $\# {\cal B}(\rho)$ and $\# {\cal B}(\rho_0)$ are bounded independently of $h$, $\rho$ and $\rho_0$ if $h \le \min\{ \frac{1}{2} (M-1) c, 1 \} \cdot \min\{ h_0(\rho), h_0(\rho_0) \}$ and that $\| \tilde{r}\|_{L^{\infty}(0,L)} + \| \tilde{R}\|_{L^{\infty}(0,L)} \le C$ by Lemma \ref{lemma:tR-PW} we have that 
\begin{align*}
  &\| \tilde{r}^{h}(\rho) - \tilde{r}^{h}(\rho_0) \|_{L^2(0,L)}^2 + \| \tilde{R}^{h}(\rho) - \tilde{R}^{h}(\rho_0) \|_{L^2(0,L)}^2 \\ 
  &~~\le \int_{(\frac{h}{2}, (N - \frac{1}{2})h) \setminus ( I^{\cal B}(\rho) \cup I^{\cal B}(\rho_0) )} 
        \big| \tilde{r}^{h}(\rho)(x_1) - \tilde{r}^{h}(\rho_0)(x_1) \big|^2 \\ 
  &\qquad\qquad\qquad\qquad\qquad\qquad\qquad + \big|\tilde{R}^{h}(\rho)(x_1) - \tilde{R}^{h}(\rho_0)(x_1) \big|^2 \, dx_1 + C h \\ 
  &~~\le h \sum_{a: Q_a \in {\cal G}(\rho) \cap {\cal G}(\rho_0)} 
         \big( \hat{C}(\rho) + \hat{C}(\rho_0) \big)^2 h\big) + Ch \\ 
  &~~\le C \big( \hat{C}(\rho) + \hat{C}(\rho_0) \big)^2 h + Ch   
\end{align*} 
by \eqref{eq:rRrhoeho0} and \eqref{eq:sum-epsa-bound} for $h$ small enough. So if $\tilde{r}^h(\rho_0) \to \bar{r}$ and $\tilde{R}^h(\rho_0) \to \bar{R}$ in $L^2$ then also $\tilde{r}^h(\rho) \to \bar{r}$ and $\tilde{R}^h(\rho) \to \bar{R}$. 
\end{proof}

\subsection{Compactness of bounded energy sequences}

Finally, we concern ourselves with the convergence of the rescaled deformations $y^h$ and give the proof of Theorem \ref{theo:compactness}. 

\begin{proof}[Proof of Theorem \ref{theo:compactness}.] 
Suppose that $y^h \in {\cal A}^h$ satisfies $I^h(y^h) \le C$. 

Choosing a suitable subsequence (not relabeled), by Lemma \ref{lemma:tilde-R-compact} (with $\rho$ small enough) we may assume that $\tilde{r}^h \to \bar{r} \in \PW^{1,2}((0,L); \R^2)$ and $\tilde{R}^h \to \bar{R} \in \PW^{1,2}((0,L); \SO(2))$ in $L^2$. We will now show that this implies $y^h \to \bar{r}$ and $\nabla_h y^h \to \bar{R}$ strongly in $L^2$. 

Using that $|y^h|, |\nabla_h y^h| \le M$, $|\tilde{r}^h| \le C$ by Lemma \ref{lemma:tR-PW}, $|\tilde{R}^h| \le C$ by construction and that $|I^{\cal B}| \le C h$ by Lemma \ref{lemma:B-bounds} we find 
\begin{align*}
  &\| y^h - \tilde{r}^h \|_{L^2(\Omega; \R^2)}^2 + \| \nabla_h y^h - \tilde{R}^h \|_{L^2(\Omega; \R^{2 \times 2})}^2 \\ 
  &~~ \le \int_{(\frac{h}{2}, (N - \frac{1}{2})h) \setminus I^{\cal B}} \int_{-\frac{1}{2} + \bar{c} \rho}^{\frac{1}{2} - \bar{c} \rho} 
       \big| y^h(x_1,x_2) - \tilde{r}^{h}(x_1) \big|^2 \\ 
  &\qquad\qquad\qquad\qquad\qquad\qquad + \big| \nabla_h y^h(x_1,x_2) - \tilde{R}^{h}(x_1) \big|^2 \, dx_1 \, dx_2 + C \rho + C h \\ 
  &~~ \le h \sum_{a : Q_a \in {\cal G}} \int_{Q_{a,\rho}} h^2 | w^h(x) - h^{-1} \tilde{r}(hx_1) |^2 + | \nabla w^h(x) - \tilde{R}(hx_1) |^2 \, dx + C \rho + C h \\ 
  &~~ \le C h \sum_{a : Q_a \in {\cal G}} \int_{Q_{a,\rho}} h^2 | w^h(x) - r_{a,1}(x) |^2 + | \nabla w^h(x) - R_{a,1} |^2 \, dx + C \rho + C h, 
\end{align*} 
where in the last step we have used \eqref{eq:tr-ra1-full-est} and \eqref{eq:tR-Ra1-full-est}. 

Setting 
\begin{align}\label{eq:bar-w}
\bar{w}_a(x) 
  = \begin{cases} 
       \hat{w}_{a}(x) &\mbox{if } x \in P_{a,1}, \\ 
       r_{a,1}(x) &\mbox{otherwise}, 
    \end{cases} 
\end{align}
we also have 
\begin{align*}
  \| \nabla \bar{w}_a - R_{a,1} \|_{L^2(Q_{a,\rho})}^2 
  = \| \nabla \hat{w}_a - R_{a,1} \|_{L^2(P_{a,1})}^2 
  \le \| \nabla \hat{u}_a \|_{L^2(Q_{a,\rho})}^2 
  \le \hat{C} \eps_a^{\frac{9}{10}} 
  \le \hat{C} h^{\frac{9}{10}}
\end{align*}
and 
\begin{align*}
  \| \bar{w}_a - r_{a,1} \|_{L^2(Q_{a,\rho})}^2 
  = \| \hat{w}_a - r_{a,1} \|_{L^2(P_{a,1})}^2 
  \le \| \hat{u}_a \|_{L^2(Q_{a,\rho})}^2 
  \le \hat{C} h 
\end{align*}
by \eqref{eq:ua-estimates} and \eqref{eq:sum-epsa-bound}, so that 
\begin{align*}
  &\int_{Q_{a,\rho}} h^2 | w^h(x) - h^{-1} \tilde{r}(hx_1) |^2 + | \nabla w^h(x) - \tilde{R}(hx_1) |^2 \, dx \\ 
  &~~ \le C \int_{Q_{a,\rho}} h^2 | w^h(x) - \bar{w}_a |^2 + | \nabla w^h(x) - \nabla \bar{w}_a |^2 \, dx + \hat{C} h^{\frac{9}{10}}. 
\end{align*}
In order to further estimate this quantity we note that by \eqref{eq:Per-Pa1-est} and \eqref{eq:Cacc-per-est}
\begin{align*}
  |Q_{a,\rho} \setminus P_{a,1}|  
  \le C \big( {\cal H}^1(J_{w^h} \cap Q_a) + C \rho \big)^2 
  \le C ( h^{-1} \eps_a + C \rho)^2 
  \le C h^{-1} \eps_a + C \rho 
\end{align*}
and, by \eqref{eq:Va-est} and \eqref{eq:sum-epsa-bound}, $|Q_{a,\rho} \setminus V_a| \le C h^{-1} \rho \eps_a \le C \rho$. Since $h w^h, \nabla w^h$ by assumption and $h \bar{w}_a, \nabla \bar{w}_a$ by Theorem \ref{theo:quant-rig} are bounded, we find 
\begin{align*}
  &\int_{Q_{a,\rho}} h^2 | w^h(x) - \bar{w}_a |^2 + | \nabla w^h(x) - \nabla \bar{w}_a |^2 \, dx \\ 
  &~~ \le \int_{P_{a,1} \cap V_a} h^2 | w^h(x) - \bar{w}_a |^2 + | \nabla w^h(x) - \nabla \bar{w}_a |^2 \, dx + C h^{-1} \eps_a + C \rho. 
\end{align*}
As by \eqref{eq:Va-est} we also have 
\begin{align*}
  &\int_{P_{a,1} \cap V_a} h^2 | w^h(x) - \bar{w}_a |^2 + | \nabla w^h(x) - \nabla \bar{w}_a |^2 \, dx \\ 
  &~~= \int_{P_{a,1} \cap V_a} h^2 | w^h(x) - \hat{w}_a |^2 + | \nabla w^h(x) - \nabla \hat{w}_a |^2 \, dx 
    \le C \rho \eps_a 
    \le C \rho h, 
\end{align*}
we may conclude that 
\begin{align*}
  \int_{Q_{a,\rho}} h^2 | w^h(x) - h^{-1} \tilde{r}(hx_1) |^2 + | \nabla w^h(x) - \tilde{R}(hx_1) |^2 
  \le C h^{-1} \eps_a + C \rho + \hat{C} h^{\frac{9}{10}}.
\end{align*}
Summing over $a$ we thus obtain from \eqref{eq:sum-epsa-bound}
\begin{align*}
  &\| y^h - \tilde{r}^h \|_{L^2(\Omega; \R^2)}^2 + \| \nabla_h y^h - \tilde{R}^h \|_{L^2(\Omega; \R^{2 \times 2})}^2 \\ 
  &~~\le h \sum_{a \in {\cal G}} \Big( C h^{-1} \eps_a + C \rho + \hat{C} h^{\frac{9}{10}} \Big) + C \rho + C h 
    \le \hat{C} h^{\frac{9}{10}} + C \rho.  
\end{align*} 
Since $\rho > 0$ was arbitrary, we have indeed shown that, in $L^2$ and hence in $L^1$ $\lim_{h  \to 0} y^h = \lim_{h  \to 0} \tilde{r}^h = \bar{r} =: y \in \PW^{1,2}((0,L); \R^2)$, where $y = \bar{r}$ is a function of $x_1$ only, and $\lim_{h  \to 0} \nabla_h y^h = \lim_{h  \to 0} \tilde{R}^h = \bar{R} \in \PW^{1,2}((0,L); \SO(2))$ in $L^2$.

By Lemma \ref{lemma:tR-PW} we then have $\partial_1 y = \bar{R} \mathbf{e}_1 \in \PW^{1,2}((0,L); \R^2)$, so that $y \in \PW^{2,2}((0,L); \R^2)$ with $|\partial_1 y| = |\bar{R} \mathbf{e}_1| = 1$. As $\| y \|_{L^{\infty}} \le \liminf_{h \to 0} \| y^h \|_{L^{\infty}} \le M$, we indeed obtain $y \in {\cal A}$. Moreover, since $\bar{R} \in \SO(2)$ and so $\bar{R} \mathbf{e}_2 = (\partial_1 y)^{\perp}$, we also have that $\nabla_h y^h \to \big( \partial_1 y \mid (\partial_1 y)^{\perp} \big)$ in $L^2$. 
\end{proof}

\begin{rem}\label{rem:crack-localization}
The convergence $y^h \to y \in {\cal A}$ could have alternatively been proved with the help of Theorem \ref{theo:SBV-compactness} and the observation that the energy bound \eqref{eq:Ihyh-bound} implies $\partial_2 y = 0$ and $\nu_2(y) = 0$ and hence $D_2 y = 0$ which would render the approximation by $\tilde{r}$ unnecessary. Our approach is more direct and in fact only needs the SBV compactness theorem in its considerably more elementary one-dimensional version given in Theorem \ref{theo:PW-compactness}. The main advantage of the proof presented here is that, as we will make precise in the following corollary, it shows that portions of the beam covered by rectangles in ${\cal G}$ do not contribute to the limiting crack set although these rectangles might contain cracks of almost length $1$. This is interesting from a physical perspective as it shows that only strongly localized cracks in $y^h$ can cause macroscopic fracture whereas smaller cracks separated by a distance at least $h$ are healed in the limit $h \to 0$. (Note that we also could have chosen arbitrary thin rectangles which are translates of $(0,c) \times (-\frac{1}{2}, \frac{1}{2})$, so that indeed a crack that is not healed must be concentrated on a set whose projection onto the $x_1$-axis is much smaller than $h$.) Also in Section \ref{sec:Gamma} we will benefit from this approach when we need to give a sharp bound on the crack energy in $y^h$ from below in terms of both $J_y$ and $J_{y'}$.  
\end{rem}

\begin{corollary}\label{cor:jump-in-B}
Suppose that $y^h$ with $I^h(y^h) \le C$ satisfy $y^h \to y$ in $L^1$. Then, for a not relabeled subsequence, $h I^{\cal B}$ converges to a finite subset of $[0, L]$ which contains $J_y \cup J_{y'}$.  
\end{corollary}

\begin{proof} 
The preceding proof of Theorem \ref{theo:compactness} gives $J_y \cup J_{y'} = J_{\bar{r}} \cup J_{\bar{R}}$. The claim now immediately follows from Lemma \ref{lemma:tilde-R-compact}(iii) as by construction $J_{\tilde{r}^h} \cup J_{\tilde{R}^h} \subset I^{\cal B}$, which converges to a finite set due to $\#{\cal B} \le C$. 
\end{proof} 

\section{Energy estimates, infinitesimal strain and $\mathbf{\Gamma}$-convergence}\label{sec:Gamma}

While the construction of recovery sequences in Theorem \ref{theo:Gamma-convergence}(ii) will be rather straightforward, the main focus of this section will be the proof of the $\Gamma$-$\liminf$ inequality in Theorem \ref{theo:Gamma-convergence}(i). To this end, we will need two preparatory steps. First, we will provide an energy estimate for the elastic energy contribution on good rectangles in terms of a suitable strain measure to be introduced below. Second, we need to identify the limiting behavior of this strain measure as $h \to 0$.

\subsection{Elastic energy estimates} 

Throughout this paragraph we again assume that $y^h \in {\cal A}^h$ is a bounded energy sequence of deformations verifying $I^h(y^h) \le C$ so that the rescaled deformations $w^h(x) = h^{-1} y^h(h x_1, x_2)$ satisfy \eqref{eq:energy-bound}. $R_{a,1}, c_{a,1}, r_{a,1}, \hat{w}_a$ and $\tilde{R}$ are as in the previous section.

\subsubsection*{Estimates on a single cell} 

Our first aim is provide an asymptotically exact estimate from below on the energy $\int_{Q_a} W(\nabla w^h) + h {\cal H}^1 (J_{w^h} \cap Q_a)$ on a single rectangle $Q_a$ in case $Q_{a-1}, Q_a, Q_{a+1} \in {\cal G}$. In order to do so we will first pass to the modified deformation $\bar{w}_a$ defined in \eqref{eq:bar-w} and let 
\begin{align}\label{eq:bu-def}
  \bar{u}_a(x) 
  := \bar{w}_a(x) - r_{a,1}(x)
  = \begin{cases} 
       \hat{u}_{a}(x) &\mbox{if } x \in P_{a,1} \\ 
       0 &\mbox{otherwise}. 
    \end{cases} 
\end{align}
on $Q_{a}$ for $x \in Q_{a}$ if $Q_a \in {\cal G}$. 

Obviously we have 
\begin{align}\label{eq:en-w-wa-est}
  \int_{Q_{a,\rho}} W(\nabla \bar{w}_a)
  \le \int_{Q_{a,\rho}} W(\nabla \hat{w}_a)
  \le \int_{Q_a} W(\nabla w^h) + h {\cal H}^1 (J_{w^h} \cap Q_a) + C \rho \eps_a 
\end{align}
by \eqref{eq:wa-elast-est} and 
\begin{align}\label{eq:crack-bw}
  {\cal H}^1(J_{\bar{w}_a} \cap Q_{a,\rho}) 
  \le C h^{-1} \eps_a 
\end{align}
by \eqref{eq:wa-crack-est}, \eqref{eq:Cacc-per-est}. Also note that due to \eqref{eq:ua-estimates} 
\begin{align}\label{eq:bu-ests}
  \| \bar{u}_a \|^2_{L^2(Q_{a,\rho})} \le \hat{C} \eps_a, \quad 
  \| e(R^T_{a,1} \nabla \bar{u}_a) \|_{L^2(Q_{a,\rho})}^2 \le \hat{C} \eps_a, \quad 
  \| \nabla \bar{u}_a \|^2_{L^2(Q_{a,\rho})} \le \hat{C}\eps_a^{\frac{9}{10}}. 
\end{align}

We now proceed to prove a lower bound on $h^{-1} \int_{Q_{a,\rho}} W(\nabla \bar{w}_a)$. In view of the global energy estimate to be obtained below it turns out to be insufficient to directly linearize $\bar{w}_a$ around $R_{a,1}$. Instead we have to modify both $R_{a,1}$ and $\nabla \bar{u}_{a}$ in a way which allows for gluing together the contributions on overlapping rectangles without introducing too much fracture. To this end, we first introduce a modification of $\tilde{R}$ which takes values in $\SO(2)$ as follows. Define $\hat{R} = \operatorname{Proj}_{\SO(2)} \tilde{R}$, where $\operatorname{Proj}_{\SO(2)}$ denotes the orthogonal projection of $\R^{2 \times 2}$ onto $\SO(2)$, which is uniquely defined and smooth in a neighborhood of $\SO(2)$. As an immediate consequence of Lemma \ref{lemma:tR-PW} we have 
\begin{align}\label{eq:ttR-tR-est}
  \| \hat{R} - \tilde{R} \|_{L^{\infty}} \le \hat{C} \sqrt{h}. 
\end{align}
In fact, by definition of $\hat{R}$ and \eqref{eq:tR-Ra1-est}, we even have 
\begin{align}\label{eq:ttRRa-Id-est}
\begin{split}
  |\hat{R}^T(h\,\cdot) R_{a,1} - \Id | 
  &\le |\hat{R}(h\,\cdot) - \tilde{R}(h\,\cdot)| + |\tilde{R}(h\,\cdot) - R_{a,1}| \\ 
  &\le 2  |\tilde{R}(h\,\cdot) - R_{a,1}| 
  \le \hat{C} \sqrt{\eps_{a-1} + \eps_a + \eps_{a+1}} 
\end{split}
\end{align}
on $Q_a$. As $\R^{2 \times 2}_{\rm sym}$ is the normal space of $\SO(2)$ at $\Id$, this entails the stronger estimate 
\begin{align}\label{eq:sym-ttR-tR-est}
  |e(\hat{R}^T(h\,\cdot) R_{a,1}) - \Id| \le \hat{C} (\eps_{a-1} + \eps_a + \eps_{a+1}) 
\end{align}
on the symmetric part of $\hat{R}^T(h\,\cdot) R_{a,1} - \Id$. 

Recall that ${\cal Q}$ denotes the Hessian of $W$ at $\Id$. 
\begin{lemma}\label{lemma:en-single-cell}
Let $\chi_a$ be the characteristic function of the set $\{ x \in Q_{a,\rho} : |\nabla \bar{u}_a| \le h^{\frac{4}{5}} \}$. Then 
\begin{align*}
  \int_{Q_{a,\rho}} W( \nabla \bar{w}_a ) 
  &\ge \int_{Q_{a,\rho}} \frac{1}{2} {\cal Q} \big( \chi_a e(\tilde{R}^T(h\,\cdot) \nabla \bar{u}_a) \big) - \hat{C} (\eps_{a-1} + \eps_a + \eps_{a+1}) h^{\frac{1}{2}} - \hat{C} h^{\frac{21}{10}}.
\end{align*}
\end{lemma} 

\begin{proof}
Taylor expanding $W$ around $\Id$ as 
\begin{align}\label{eq:Taylor-exp} 
  W(\Id + X) 
  = \frac{1}{2} {\cal Q}(X) + \omega(X), 
\end{align}
where $|\omega(X)| \le C |X|^3$ for $X$ small, we compute on $Q_{a,\rho}$ (with $\tilde{R}$, $\hat{R}$ evaluated at $h x_1$ and $\bar{w}_a$, $\bar{u}_a$ evaluated at $x$) 
\begin{align}\label{eq:Wlb-i}
\begin{split}
  &W \big( \nabla \bar{w}_a \big) 
   \ge \chi_a W \big( \hat{R}^T \nabla \bar{w}_a \big) \\ 
  &~~= \chi_a W \big( \Id + \hat{R}^T R_{a,1} - \Id + \hat{R}^T \nabla \bar{u}_a \big) \\ 
  &~~= \frac{1}{2} \chi_a {\cal Q} \big( \hat{R}^T R_{a,1} - \Id + \hat{R}^T \nabla \bar{u}_a \big) 
     + \chi_a \omega \big( \hat{R}^T R_{a,1} - \Id + \hat{R}^T \nabla \bar{u}_a \big) \\ 
  &~~= \frac{1}{2} \chi_a {\cal Q} \big( e(\hat{R}^T R_{a,1} - \Id) + e(\hat{R}^T \nabla \bar{u}_a) \big) 
     + \chi_a \omega \big( \hat{R}^T R_{a,1} - \Id + \hat{R}^T \nabla \bar{u}_a \big), 
\end{split}
\end{align}
where we have used that $W$ is invariant under rotations and ${\cal Q}$ vanishes on $\R^{2 \times 2}_{\rm skew}$. 

Recalling \eqref{eq:ttRRa-Id-est} and \eqref{eq:sum-epsa-bound} we can estimate the error term by 
\begin{align}\label{eq:Wlb-ii}
\begin{split}
  &\chi_a |\omega \big( \hat{R}^T R_{a,1} - \Id + \hat{R}^T \nabla \bar{u}_a \big)| 
   \le C \chi_a \big( | \hat{R} - R_{a,1} |^3 + |\nabla \bar{u}_a|^3 \big) \\ 
  &~~\le \hat{C} (\eps_{a-1} + \eps_a + \eps_{a+1})^{\frac{3}{2}} + C h^{\frac{12}{5}}  
   \le \hat{C} (\eps_{a-1} + \eps_a + \eps_{a+1}) h^{\frac{1}{2}} + C h^{\frac{12}{5}}. 
\end{split}
\end{align}
Moreover, since $|\hat{R} - \tilde{R}| \le \hat{C} \sqrt{h}$ and $|e(\hat{R}^T R_{a,1} - \Id)| \le \hat{C}(\eps_{a-1} + \eps_a + \eps_{a+1})$ by \eqref{eq:ttR-tR-est} and \eqref{eq:sym-ttR-tR-est}, respectively, we have  
\begin{align}\label{eq:Wlb-iii}
\begin{split}
  &\chi_a {\cal Q} \big( e(\hat{R}^T \nabla \bar{u}_a) + e(\hat{R}^T R_{a,1} - \Id) \big) \\ 
  &~~= \chi_a {\cal Q} \big( e(\tilde{R}^T \nabla \bar{u}_a) + e \big( (\hat{R}^T - \tilde{R}^T) \nabla \bar{u}_a \big) + e(\hat{R}^T R_{a,1} - \Id) \big) \\ 
  &~~\ge \chi_a {\cal Q} \big( e(\tilde{R}^T \nabla \bar{u}_a) \big) \\ 
  &\qquad     - C \chi_a | e(\tilde{R}^T \nabla \bar{u}_a) | \big( | e \big( (\hat{R}^T - \tilde{R}^T) \nabla \bar{u}_a \big) | + | e(\hat{R}^T R_{a,1} - \Id) | \big) \\ 
  &~~\ge \chi_a {\cal Q} \big( e(\tilde{R}^T \nabla \bar{u}_a) \big) 
       - C \chi_a | \nabla \bar{u}_a | \big( |\hat{R}^T - \tilde{R}^T| \, |\nabla \bar{u}_a | + | e(\hat{R}^T R_{a,1} - \Id) | \big) \\ 
  &~~\ge \chi_a {\cal Q} \big( e(\tilde{R}^T \nabla \bar{u}_a) \big) - \hat{C} h^{\frac{4}{5}} (h^{\frac{1}{2}} h^{\frac{4}{5}} + \eps_{a-1} + \eps_a + \eps_{a+1}). 
\end{split}
\end{align}
Combining \eqref{eq:Wlb-i}, \eqref{eq:Wlb-ii} and \eqref{eq:Wlb-iii} we find that indeed 
\begin{align*}
  \int_{Q_{a,\rho}} W( \nabla \bar{w}_a ) 
  &\ge \int_{Q_{a,\rho}} \frac{1}{2} \chi_a {\cal Q} \big( e(\tilde{R}^T \nabla \bar{u}_a) \big) - \hat{C} (\eps_{a-1} + \eps_a + \eps_{a+1}) h^{\frac{1}{2}} - \hat{C} h^{\frac{21}{10}}.
\end{align*}
\end{proof}

\subsubsection*{Global estimates} 

We now give an estimate for a connected part of $h^{-1} \Omega_h$ covered by rectangles in ${\cal G}$. In view of the fact that we later will have to identify the limiting strain on such a part, it turns out to be necessary to also modify $\bar{w}_a$. To this end, we introduce the following measure $F^{(2)}$ of the elastic strain. If $Q_{a-1}, \ldots, Q_{b+1} \in {\cal G}$ we set 
\begin{align}\label{eq:F-def}
  F^{(2)}(x) 
  = \sum_{i=a-1}^{b+1} \varphi_i(x_1) \nabla \bar{w}_{i}(x) 
  \quad\mbox{and}\quad 
  G^{(2)}(hx_1,x_2) 
  = \frac{e \big( \tilde{R}^T(h x_1) F^{(2)}(x) - \Id \big) }{h} 
\end{align} 
for $x \in Q_{a,\rho} \cup \ldots \cup Q_{b,\rho}$ with  the same partition of unity $(\varphi_i)$ as in the definition of $\tilde{R}$, cf.\ \eqref{eq:tR-def}. Also let $\chi_{a,b}$ be the characteristic function of the set 
\begin{align}\label{eq:chi-a-b-def}
  \big\{ x \in Q_{a-1,\rho} \cup \ldots \cup Q_{b+1,\rho} : |\nabla \bar{u}_i(x)| \le h^{\frac{4}{5}} \mbox{ if } x \in Q_{i,\rho} \big\}. 
\end{align}

\begin{lemma}\label{lemma:en-cell-string}
Suppose that $Q_{a-1}, \ldots, Q_{b+1} \in {\cal G}$. 
Then 
\begin{align*}
  &\sum_{i=a}^{b} \int_{Q_i} W \big( \nabla w^h(x) \big) \, dx + h {\cal H}^1 (J_{w^h} \cap Q_i) \\ 
  &~~\ge \frac{h}{2} \int_{h(a-1)}^{hb} \int_{-\frac{1}{2} + \bar{c} \rho}^{\frac{1}{2} - \bar{c} \rho} {\cal Q} \big( \chi_{a,b}(h^{-1}x_1, x_2) G^{(2)}(x) \big) \, dx 
    - \hat{C} h^{\frac{11}{10}} - C \rho h. 
\end{align*}
\end{lemma}

\begin{proof}
By \eqref{eq:en-w-wa-est} and Lemma \ref{lemma:en-single-cell} we have 
\begin{align*}
  &\sum_{i=a}^{b} \int_{Q_i} W(\nabla w^h) + h {\cal H}^1 (J_{w^h} \cap Q_i) 
   \ge \sum_{i=a}^{b} \int_{Q_{i,\rho}} W(\nabla \bar{w}_i) - C \rho \eps_i \\ 
  &~~\ge \sum_{i=a}^{b} \bigg( \int_{Q_{i,\rho}} \frac{1}{2} \chi_{a,b} {\cal Q} \big( e(\tilde{R}^T(h\,\cdot) \nabla \bar{u}_i) \big) 
    - \hat{C} (\eps_{i-1} + \eps_i + \eps_{i+1}) h^{\frac{1}{2}} - \hat{C} h^{\frac{21}{10}}  - C \rho \eps_i \bigg). 
\end{align*} 
In order to further estimate this in terms of $F^{(2)}$ we first note that (again abbreviating $\tilde{R}(h\,\cdot)$ by $\tilde{R}$) 
$$ \tilde{R}^T F^{(2)} 
   = \tilde{R}^T \nabla \bar{w}_i 
   = R_{i,1}^T \nabla \bar{w}_i 
   = \Id + R_{i,1}^T \nabla \bar{u}_i 
   = \Id + \tilde{R}^T \nabla \bar{u}_i $$ 
on $Q_{i,\rho} \setminus (Q_{i-1,\rho} \cup Q_{i+1,\rho})$. On the other hand, if $j = i-1$ or $j=i+1$, then on $Q_{i,\rho} \cap Q_{j,\rho}$
\begin{align*}
  \tilde{R}^T F^{(2)} 
  &= \tilde{R}^T (\varphi_i \nabla \bar{w}_i + \varphi_j \nabla \bar{w}_j) \\ 
  &= \tilde{R}^T \big( \varphi_i (R_{i,1} + \nabla \bar{u}_i) + \varphi_j (R_{j,1} + \nabla \bar{u}_j) \big) \\ 
  &= \tilde{R}^T \tilde{R} + \varphi_i \tilde{R}^T \nabla \bar{u}_i + \varphi_j \tilde{R}^T \nabla \bar{u}_j \\ 
  &= \Id + \varphi_i \tilde{R}^T \nabla \bar{u}_i + \varphi_j \tilde{R}^T \nabla \bar{u}_j. 
\end{align*}
By convexity of ${\cal Q}$ we therefore have 
\begin{align*}
  {\cal Q} \big( e( \tilde{R}^T F^{(2)} - \Id) \big) 
  &= {\cal Q} \big(\varphi_i e(\tilde{R}^T \nabla \bar{u}_i) + \varphi_j e(\tilde{R}^T \nabla \bar{u}_j) \big) \\ 
  &\le \varphi_i {\cal Q} \big(e(\tilde{R}^T \nabla \bar{u}_i) \big) + \varphi_j {\cal Q} \big( e(\tilde{R}^T \nabla \bar{u}_j) \big) \\ 
  &\le {\cal Q} \big(e(\tilde{R}^T \nabla \bar{u}_i) \big) + {\cal Q} \big( e(\tilde{R}^T \nabla \bar{u}_j) \big) 
\end{align*} 
on $Q_{i,\rho} \cap Q_{j,\rho}$. Combining with the previous estimates we arrive at 
\begin{align*}
  &\sum_{i=a}^{b} \int_{Q_i} W \big( \nabla w^h(x) \big) \, dx + h {\cal H}^1 (J_{w^h} \cap Q_i) \\ 
  &~~\ge  \frac{1}{2} \int_{a-1}^{b} \int_{-\frac{1}{2} + \bar{c} \rho}^{\frac{1}{2} - \bar{c} \rho} \chi_{a,b}(x) {\cal Q} \big( e(\tilde{R}^T(hx_1) F^{(2)}(x)) - \Id \big) \, dx \\ 
  &~~\qquad\qquad   - \sum_{i = a}^b \big( \hat{C} (\eps_{i-1} + \eps_i + \eps_{i+1}) h^{\frac{1}{2}} - \hat{C} h^{\frac{21}{10}}  - C \rho \eps_i \big) \\ 
  &~~\ge \frac{h^2}{2} \int_{a-1}^{b} \int_{-\frac{1}{2} + \bar{c} \rho}^{\frac{1}{2} - \bar{c} \rho} \chi_{a,b}(x) {\cal Q} \big( G^{(2)}(h x_1, x_2) \big) \, dx 
    - \hat{C} h^{\frac{3}{2}} - \hat{C} h^{\frac{11}{10}} -C \rho \sum_{j = 1}^N \eps_j \\ 
  &~~\ge \frac{h}{2} \int_{h(a-1)}^{hb} \int_{-\frac{1}{2} + \bar{c} \rho}^{\frac{1}{2} - \bar{c} \rho} {\cal Q} \big( \chi_{a,b}(h^{-1}x_1, x_2) G^{(2)}(x) \big) \, dx 
    - \hat{C} h^{\frac{11}{10}} - C \rho h  
\end{align*}
by \eqref{eq:sum-epsa-bound}. 
\end{proof}

\subsection{Deformation interpolation and limiting strain}

In the following we will have to identify, on parts covered by rectangles in ${\cal G}$, the limiting behavior of $G^{(2)}$, cf.\ \eqref{eq:F-def}. In particular, it will be essential to ensure that in passing from the piecewisely defined displacements $\nabla \bar{u}_a$ to the global quantity $F^{(2)}$ the limiting infinitesimal strain (of typical order $h$ on each rectangle) is measured sufficiently accurately by $F^{(2)}$. The main difficulty in doing so is caused by the fact that $F^{(2)}$ itself is not a gradient. For this reason we now also introduce an interpolation $\tilde{w}$ of the $\bar{w}_a$, cf.\ \eqref{eq:bar-w}. We will then proceed in two steps. First we will prove that, in a suitable weak sense, $F^{(2)}$ is very close to $\nabla \tilde{w}$. Second we will show that the relevant limiting entries of $F^{(2)}$ can be recovered with the help of an $SBV$ closure argument applied to an auxiliary function which arises from $w^h$ by `undoing the rotation $\tilde{R}$' and suitably rescaling both its image and its preimage.

If $Q_{a-1}, \ldots, Q_{b+1} \in {\cal G}$ we define an interpolation $\tilde{w} = \tilde{w}^h \in SBV \big(Q_{a,\rho} \cup \ldots \cup Q_{b,\rho} ; \R^2 \big)$ by setting 
\begin{align*}
  \tilde{w}(x) 
  = \sum_{i=a-1}^{b+1} \varphi_i(x_1) \bar{w}_{i}(x) 
\end{align*}
with the same partition of unity $(\varphi_a)_{a \in \N}$ as in the definition of $\tilde{R}$ and $F^{(2)}$, cf.\ \eqref{eq:tR-def} and \eqref{eq:F-def}. The jump set $J_{\tilde{w}}$ of $\tilde{w}$ then satisfies $J_{\tilde{w}} \subset \bigcup_{i=a-1}^{b+1} (J_{\bar{w}_{i,1}} \cap Q_{i,\rho})$ and thus can readily be estimated by  
\begin{align}\label{eq:jump-tw-est}
  {\cal H}^1 \big( J_{\tilde{w}} \cap (Q_{a,\rho} \cup \ldots \cup Q_{b,\rho}) \big)
  \le \sum_{i=a-1}^{b+1} {\cal H}^1(J_{\bar{w}_{i}} \cap Q_{i,\rho})
  \le h^{-1} \sum_{i=a-1}^{b+1} \eps_i 
  \le C 
\end{align}
with the help of \eqref{eq:crack-bw} and \eqref{eq:sum-epsa-bound}. With $F^{(2)}$ as in \eqref{eq:F-def} and appropriately defined $F^{(1)}$ the absolutely continuous part of the gradient of $\tilde{w}^h$ splits as 
\begin{align}\label{eq:nabla-tw-split}
  \nabla \tilde{w} 
  = \sum_{i=a-1}^{b+1} \varphi_i' (\bar{w}_i \mid 0) + \sum_{i=a-1}^{b+1} \varphi_i \nabla \bar{w}_i 
  =: F^{(1)} + F^{(2)}. \end{align}

\subsubsection*{Gradient estimates}

Our main task is to show that, to leading order $h$, $\nabla \tilde{w}^h$ can be approximated in a suitable sense by $F^{(2)}$ only. For this we first need to give a precise estimate on the difference of the $\bar{w}_a$ on overlapping rectangles.
\begin{lemma}\label{lemma:F1-local-est} 
Let $1 \le p < 2$. Suppose that $Q_a, Q_b$, $b = a+1$, are overlapping rectangles in ${\cal G}$. Then 
$$ \| \bar{w}_a - \bar{w}_b \|_{L^p(Q_{a,\rho} \cap Q_{b,\rho})}^p 
   \le C \rho^{\frac{p}{2}} ( \eps_a^{\frac{p}{2}} + \eps_b^{\frac{p}{2}}) + \hat{C} h^{\frac{p-2}{2}} \, (\eps_a + \eps_b). $$ 
\end{lemma}

\begin{proof}
On $V = V_a \cap P_{a,1} \cap V_b \cap P_{b,1}$ we have 
\begin{align}\label{eq:wa-w-V}
 \| \bar{w}_a - w^h \|_{L^p(V)}^p 
   \le C \| \bar{w}_a - w^h \|_{L^2(V)}^p 
   \le  C (\rho \eps_a)^{\frac{p}{2}}. 
\end{align}
by \eqref{eq:bar-w} and \eqref{eq:Va-est}.

In order to estimate this difference on the remaining part we first note that by \eqref{eq:Per-Pa1-est} and \eqref{eq:Cacc-per-est} we have 
\begin{align*}
  | (Q_{a,\rho} \cap Q_{b,\rho}) \setminus (P_{a,1} \cap P_{b,1}) | 
  &\le | Q_{a,\rho} \setminus P_{a,1} | + | Q_{b,\rho} \setminus P_{b,1} | \\ 
  &\le C \big( (\mathrm{Per}(P_{a,1}, Q_{a,\rho}))^2 + (\mathrm{Per}(P_{b,1}, Q_{b,\rho}))^2 \big) \\ 
  &\le C h^{-1} (\eps_a + \eps_b). 
\end{align*}
By \eqref{eq:Va-est} the same estimate holds for $| (Q_{a,\rho} \cap Q_{b,\rho}) \setminus (V_a \cap V_b) |$ and hence $| (Q_{a,\rho} \cap Q_{b,\rho}) \setminus V| \le C h^{-1} (\eps_a + \eps_b)$. But then H{\"o}lder's inequality gives 
\begin{align*}
  \int_{(Q_{a,\rho} \cap Q_{b,\rho}) \setminus V} |\bar{w}_a - r_{a,1}|^p 
  &\le   | (Q_{a,\rho} \cap Q_{b,\rho}) \setminus V |^{\frac{2-p}{2}} \, \bigg( \int_{P_{a,1}} |\bar{w}_a - r_{a,1}|^2 \bigg)^{\frac{p}{2}} \\ 
  &\le C \big( h^{-1} (\eps_a + \eps_b) \big)^{\frac{2-p}{2}} \, \| \bar{u}_a \|_{L^2(Q_{a,\rho})}^p 
   \le \hat{C} h^{\frac{p-2}{2}} \, (\eps_a + \eps_b), 
\end{align*} 
where we have used that $\| \bar{u}_a \|_{L^2(Q_{a,\rho})}^p \le \hat{C} \eps_a^{\frac{p}{2}} \le \hat{C} (\eps_a + \eps_b)^{\frac{p}{2}}$ according to \eqref{eq:bu-ests}. 

Together with \eqref{eq:wa-w-V} this shows that 
\begin{align}\label{eq:wa-w}
  \| \bar{w}_a - \chi_{V} w^h - \chi_{(Q_{a,\rho} \cap Q_{b,\rho}) \setminus V} r_{a,1} \|_{L^p(Q_{a,\rho} \cap Q_{b,\rho})}^p 
  \le C (\rho \eps_a)^{\frac{p}{2}} + \hat{C} h^{\frac{p-2}{2}} \, (\eps_a + \eps_b). 
\end{align} 
By Lemma \ref{lemma:RaRb-est} we have $\| r_{a,1} - r_{b,1} \|_{\infty}^2 \le \hat{C}(\eps_a + \eps_b)$ on $Q_a$ and so, similarly,  
\begin{align*}
  \int_{(Q_{a,\rho} \cap Q_{b,\rho}) \setminus V} |r_{a,1} - r_{b,1}|^p 
  &\le  | (Q_{a,\rho} \cap Q_{b,\rho}) \setminus V |^{\frac{2-p}{2}} \, \bigg( \int_{Q_{a,\rho}} |r_{a,1} - r_{b,1}|^2 \bigg)^{\frac{p}{2}} \\ 
  &\le \hat{C} \big( h^{-1} (\eps_a + \eps_b) \big)^{\frac{2-p}{2}} \, (\eps_a + \eps_b)^{\frac{p}{2}} 
   = \hat{C} h^{\frac{p-2}{2}} \, (\eps_a + \eps_b).
\end{align*} 
So with \eqref{eq:wa-w} we also obtain 
\begin{align}\label{eq:wa-w-Rb}
  \| \bar{w}_a - \chi_{V} w^h - \chi_{(Q_{a,\rho} \cap Q_{b,\rho}) \setminus V} r_{b,1} \|_{L^p(Q_{a,\rho} \cap Q_{b,\rho})}^p 
  \le C (\rho \eps_a)^{\frac{p}{2}} + \hat{C} h^{\frac{p-2}{2}} \, (\eps_a + \eps_b). 
\end{align} 
As analogously to \eqref{eq:wa-w} we have 
\begin{align}\label{eq:wb-w}
  \| \bar{w}_b - \chi_{V} w^h - \chi_{(Q_{a,\rho} \cap Q_{b,\rho}) \setminus V} r_{b,1} \|_{L^p(Q_{a,\rho} \cap Q_{b,\rho})}^p 
  \le C (\rho \eps_b)^{\frac{p}{2}} + \hat{C} h^{\frac{p-2}{2}} \, (\eps_a + \eps_b), 
\end{align} 
we may finally combine \eqref{eq:wb-w} and \eqref{eq:wa-w-Rb} to see that indeed 
$$  \| \bar{w}_a - \bar{w}_b \|_{L^p(Q_{a,\rho} \cap Q_{b,\rho})}^p 
   \le C \rho^{\frac{p}{2}} ( \eps_a^{\frac{p}{2}} + \eps_b^{\frac{p}{2}}) + \hat{C} h^{\frac{p-2}{2}} \, (\eps_a + \eps_b). $$ 
\end{proof}

We can now prove the following global estimate on $F^{(1)}$, cf.\ \eqref{eq:nabla-tw-split}. 
\begin{lemma}\label{lemma:F1-global-est} 
Let $1 \le p < 2$. 
Suppose that $Q_{a-1}, \ldots, Q_{b+1} \in {\cal G}$. Then 
$$ \| F^{(1)} \|_{L^p(Q_{a,\rho} \cup \ldots \cup Q_{b,\rho})}^p 
   \le ( C \rho^{\frac{p}{2}} + \hat{C} h^{1 - \frac{p}{2}} )  h^{p-1}. $$
\end{lemma} 

\begin{proof} 
Since, by construction, $\varphi_i'$ vanishes outside $(i-1-\frac{1}{2n}+\bar{c}\rho, i-1+\frac{1}{2n}-\bar{c}\rho) \cup (i-\frac{1}{2n}+\bar{c}\rho, i+\frac{1}{2n}-\bar{c}\rho)$ we have that
\begin{align}\label{eq:F-sum-lb}
\begin{split}
  \| F^{(1)} \|_{L^p(Q_{a,\rho} \cup \ldots \cup Q_{b,\rho})} ^p 
  &= \sum_{i=a-1}^{b} \| \varphi_i' \bar{w}_i + \varphi_{i+1}' \bar{w}_{i+1} \|_{L^p(Q_{i,\rho} \cap Q_{i+1,\rho})}^p \\
  &\le C \sum_{i=a-1}^{b} \| \bar{w}_i - \bar{w}_{i+1} \|_{L^p(Q_{i,\rho} \cap Q_{i+1,\rho})}^p, 
\end{split}
\end{align} 
where we have used that $\varphi_i' + \varphi_{i+1}' = 0$ and $|\varphi_i'| \le 2n$ on $(i - \frac{1}{2n}, i + \frac{1}{2n})$. With Lemma \ref{lemma:F1-local-est}, H{\"o}lder's inequality and \eqref{eq:sum-epsa-bound} we can further estimate 
\begin{align*}
  \sum_{i=a-1}^{b} \| \bar{w}_i - \bar{w}_{i+1} \|_{L^p(Q_{i,\rho} \cap Q_{i+1,\rho})}^p
  &\le C \rho^{\frac{p}{2}} \sum_{i=a-1}^{b+1} \eps_i^{\frac{p}{2}} + \hat{C} h^{\frac{p-2}{2}} \sum_{i=a-1}^{b+1} \eps_i \\ 
  &\le C \rho^{\frac{p}{2}}  h^{\frac{p-2}{2}} \bigg( \sum_{i=1}^N \eps_i \bigg)^{\frac{p}{2}}
        + \hat{C} h^{\frac{p-2}{2}} \sum_{i=1}^N \eps_i \\
  &\le C \rho^{\frac{p}{2}}  h^{p-1} 
        + \hat{C} h^{\frac{p}{2}}, 
\end{align*} 
from which, together with \eqref{eq:F-sum-lb}, the claim follows. 
\end{proof}

The second part $F^{(2)}$ of $\nabla \tilde{w}^h$ is more easily estimated as follows. 
\begin{lemma}\label{lemma:F2-global-est}  
Suppose that $Q_{a-1}, \ldots, Q_{b+1} \in {\cal G}$. Then 
$$ \| e \big( (\tilde{R}(h\,\cdot))^T F^{(2)} \big) - \Id \|_{L^2(Q_{a,\rho} \cup \ldots \cup Q_{b,\rho})}^2 
   \le \hat{C} h. $$
\end{lemma} 

\begin{proof} 
Since 
\begin{align*}
  F^{(2)}(x) 
  &= \sum_{i=a-1}^{b+1} \varphi_i(x_1) \nabla \bar{w}_i(x) \\ 
  &= \sum_{i=a-1}^{b+1} \varphi_i(x_1) ( R_{i,1} + \nabla \bar{u}_i(x) ) 
  = \tilde{R}(h x) + \sum_{i=a-1}^{b+1} \varphi_i(x_1) \nabla \bar{u}_i(x), 
\end{align*}
we have 
\begin{align*}
  &\| e \big( (\tilde{R}(h\,\cdot))^T F^{(2)} \big) - \Id \|_{L^2(Q_{a,\rho} \cup \ldots \cup Q_{b,\rho})}^2 
   \le \bigg\| \sum_{i=a-1}^{b+1} | e \big( (\tilde{R}(h\,\cdot))^T \nabla \bar{u}_i \big) | \bigg\|_{L^2(Q_{a,\rho} \cup \ldots \cup Q_{b,\rho})}^2 \\ 
  &~~ \le 2 \sum_{i=a-1}^{b+1} \| e \big( R_{i,1}^T \nabla \bar{u}_i \big) \|_{L^2(Q_{i,\rho})}^2 
       + 2 \sum_{i=a-1}^{b} \| \big( R_{i,1} - \tilde{R}(h\,\cdot) \big)^T \nabla \bar{u}_i \big) \|_{L^2(Q_{i,\rho} \cap Q_{i+1,\rho})}^2, 
\end{align*}
which indeed is bounded by $\hat{C} h$ according to \eqref{eq:bu-ests}, \eqref{eq:tR-Ra1-est} and \eqref{eq:sum-epsa-bound}. 
\end{proof}

In addition to $G^{(2)}$ from \eqref{eq:F-def}, for later use we also introduce the quantities 
\begin{align}\label{eq:G1-def}
  G^{(1)}(hx_1, x_2) 
  = \frac{e \big( \tilde{R}^T(h x_1) F^{(1)}(x) \big) }{h} 
\end{align}
and 
\begin{align}\label{eq:G-def}
  G(hx_1, x_2) 
  = \frac{e \big( \tilde{R}^T(h x_1) \nabla \tilde{w}(x) - \Id \big) }{h} 
  = G^{(1)}(hx_1, x_2) + G^{(2)}(hx_1, x_2) 
\end{align}
(cf.\ \eqref{eq:nabla-tw-split}) and record the following direct consequence of the preceding lemmas. 
\begin{lemma}\label{lemma:G-ests}
Suppose that $\Omega_{s,t} = (s, t) \times (-\frac{1}{2} + \bar{c} \rho, \frac{1}{2} - \bar{c} \rho) \subset \Omega$ is such that $\Omega_{h^{-1}s,h^{-1}t} \subset Q_a \cup \ldots \cup Q_b$ with $Q_{a-1}, \ldots, Q_{b+1} \in {\cal G}$. Let $1 \le p < 2$. Then 
$$ \| G^{(2)} \|_{L^2(\Omega_{s,t})}^2 \le \hat{C}, \qquad 
   \| G^{(1)} \|_{L^p(\Omega_{s,t})}^p \le C \rho^{\frac{p}{2}} + \hat{C} h^{1 - \frac{p}{2}}, \qquad 
   \| G \|_{L^p(\Omega_{s,t})}^p \le \hat{C}. $$
\end{lemma} 

\begin{proof} 
The first estimate is only a reformulation of Lemma \ref{lemma:F2-global-est}, the second one is a direct consequence of Lemma \ref{lemma:F1-global-est} and the last one follows from the previous two immediately. 
\end{proof}

\subsubsection*{Difference quotients estimates}

In the proof of Theorem \ref{theo:Gamma-convergence} below we will in particular need to determine the $x_2$-dependence to leading order of (the relevant part of) $F^{(2)}$. As a preparatory step we provide the following estimates on difference quotients $\Delta^{(z)} \tilde{w}$ of $\tilde{w}$, where for $z \in \R \setminus \{0\}$ and a function $f$ we set 
$$ \Delta^{(z)} f(x) 
   = \frac{f(x_1, x_2 + z) - f(x_1, x_2)}{z} $$ 
whenever both $x$ and $(x_1, x_2 + z)$ belong to its domain of definition.

\begin{lemma}\label{lemma:tw-diff-quot}
Let $Q_{a-1}, \ldots, Q_{b+1} \in {\cal G}$, $S \subset \subset Q_{a,\rho} \cup \ldots \cup Q_{b,\rho}$, $z \in \R \setminus \{0\}$ with $|z| < \dist \big( S, \partial (Q_{a,\rho} \cup \ldots \cup Q_{b,\rho}) \big)$ and $1 \le p < 2$. Then 
\begin{itemize}
\item[(i)] $\| \Delta^{(z)} \tilde{w} - \tilde{R}(h\,\cdot) \mathbf{e}_2 \|_{L^2(S)}^2 \le \hat{C} h$, 
\item[(ii)] ${\cal H}^1(J_{\Delta^{(z)} \tilde{w}} \cap S) \le C$ and 
\item[(iii)] $\| e(\tilde{R}^T(h\,\cdot) \nabla \Delta^{(z)} \tilde{w}) \|_{L^p(S)}^p \le \hat{C} h^{p - 1}$ 
\end{itemize}
for a constant $C$, independent of $a$, $b$, $z$ and $S$, and a constant $\hat{C}$, independent of $a$, $b$ and $S$. 
\end{lemma}

\begin{proof}
(i) Recalling \eqref{eq:bu-def}, similarly to the proof of Lemma \ref{lemma:F2-global-est} on $Q_{a,\rho} \cup \ldots \cup Q_{b,\rho}$ we calculate 
\begin{align*}
  \Delta^{(z)} \tilde{w}(x) 
  &= \sum_{i=a-1}^{b+1} \varphi_i(x_1) \frac{\bar{w}_i(x_1, x_2 + z) - \bar{w}_i(x_1, x_2)}{z} \\ 
  &= \sum_{i=a-1}^{b+1} \varphi_i(x_1) \Big( R_{i,1} \mathbf{e}_2 + \frac{\bar{u}_i(x_1, x_2 + z) - \bar{u}_i(x_1, x_2)}{z} \Big) \\ 
  &= \tilde{R}(h x_1) \mathbf{e}_2 + \sum_{i=a-1}^{b+1} \varphi_i(x_1) \Delta^{(z)} \bar{u}_i(x). 
\end{align*}
As a consequence we may estimate 
\begin{align*}
  \| \Delta^{(z)} \tilde{w} - \tilde{R}(h\,\cdot) \mathbf{e}_2 \|_{L^2(S)}^2 
  \le \sum_{i=a-1}^{b+1} \| \Delta^{(z)} \bar{u}_i \|_{L^2(S \cap Q_{i,\rho})}^2 
  \le \hat{C} \sum_{i=a-1}^{b+1} \| \bar{u}_i \|_{L^2(Q_{i,\rho})}^2 
  \le \hat{C} h. 
\end{align*}
by \eqref{eq:bu-ests} and \eqref{eq:sum-epsa-bound}. 
\medskip 

(ii) This follows from $J_{\Delta^{(z)} \tilde{w}} \cap S \subset J_{\tilde{w}} \cup \big( (0,-z) + J_{\tilde{w}} \big)$ and \eqref{eq:jump-tw-est}. 
\medskip 

(iii) In order to estimate symmetric part of $\tilde{R}^T(h\,\cdot) \nabla \Delta^{(z)} \tilde{w} = \tilde{R}^T(h\,\cdot) \Delta^{(z)} \nabla \tilde{w}$ we compute 
\begin{align*}
  &\| e \big( \tilde{R}^T(h\,\cdot) \Delta^{(z)} \nabla \tilde{w} \big) \|_{L^p(S)}^p 
   = \| e \big( \tilde{R}^T(h\,\cdot) \Delta^{(z)} (F^{(1)} + F^{(2)}) \big) \|_{L^p(S)}^p \\ 
  &~~\le C \| \Delta^{(z)} F^{(1)} \|_{L^p(S)}^p 
      + C \| \Delta^{(z)} e \big( \tilde{R}^T(h\,\cdot) F^{(2)} - \Id \big) \|_{L^p(S)}^p \\ 
  &~~\le \hat{C} \| F^{(1)} \|_{L^p(Q_{a,\rho} \cup \ldots \cup Q_{b,\rho})}^p + \hat{C} \| e \big( \tilde{R}^T(h\,\cdot) F^{(2)} - \Id \big) \|_{L^p(Q_{a,\rho} \cup \ldots \cup Q_{b,\rho})}^p, 
\end{align*}
where we have used that $\tilde{R}$ only depends on $x_1$ and $\Id$ does not depend on $x$ at all. Since by Lemma \ref{lemma:F1-global-est} $\| F^{(1)} \|_{L^p}^p \le ( C \rho^{\frac{p}{2}} + \hat{C} h^{1 - \frac{p}{2}} ) h^{p-1} \le \hat{C} h^{p-1}$ and by Lemma \ref{lemma:F2-global-est} also  
\begin{align*}
  \| e \big( \tilde{R}^T(h\,\cdot) F^{(2)} - \Id \big) \|_{L^p}^p 
  &\le | Q_{a,\rho} \cup \ldots \cup Q_{b,\rho} |^{\frac{2-p}{2}} \, \| e \big( \tilde{R}^T F^{(2)} - \Id \big) \|_{L^2}^p \\ 
  &\le \hat{C} h^{\frac{p - 2}{2}} h^{\frac{p}{2}} 
   = \hat{C} h^{p - 1},
\end{align*}
the assertion follows. 
\end{proof}

\subsubsection*{Rescaled strain estimates} 

We now consider the rescaled and interpolated deformation $\tilde{y}(x) = \tilde{y}^h(x) = h w^h(h^{-1} x_1, x_2)$ and the quantity $\tilde{R}^T \nabla_h \tilde{y}$ measuring the associated strain on portions of $\Omega$ covered by good rectangles. Our aim is to show that its upper left component to leading order in $h$ is linear in $x_2$. 

\begin{lemma}\label{lemma:delta-z-tw}
Suppose that $\Omega_{s,t} = (s, t) \times (-\frac{1}{2} + \bar{c} \rho, \frac{1}{2} - \bar{c} \rho) \subset \Omega$ is such that, for a sequence $h \to 0$, $\Omega_{h^{-1}s,h^{-1}t} \subset Q_a \cup \ldots \cup Q_b$ with $Q_{a-1}, \ldots, Q_{b+1} \in {\cal G}$ for any $h$. Also assume that $y^h \to y$ in $L^1(\Omega_{s,t})$. If $U \subset\subset \Omega_{s,t}$ and $z \in \R$ with $|z| < \dist(\partial U, \Omega_{s,t})$, then 
$$ h^{-1} \big( \Delta^{(z)} \tilde{R}^T \nabla_h \tilde{y}^h \big)_{11} 
   \weakly \partial_{11} y \cdot (\partial_1 y)^{\perp} $$ 
weakly in $L^p(U)$ for $1 \le p < 2$ as $h \to 0$. 
\end{lemma}

\begin{proof}
First recall that in the proof of Theorem \ref{theo:compactness} we have shown that $\tilde{R}^h \to \big( \partial_1 y \mid (\partial_1 y)^{\perp} \big)$ so that by also using Lemma \ref{lemma:tilde-R-compact} we have 
\begin{align}\label{eq:R-conv-sw}
  \tilde{R}^h \to \big( \partial_1 y \mid (\partial_1 y)^{\perp} \big) 
  \quad\mbox{and}\quad 
  (\tilde{R}^h)' \weakly 
  \big( \partial_{11} y \mid \partial_1 (\partial_1 y)^{\perp} \big)
\end{align}
in $L^2(s,t)$. 

Since $h^{-1} \Delta^{(z)} \tilde{y}^h (x) =  \Delta^{(z)} \tilde{w} (h^{-1} x_1, x_2)$ and $\Delta^{(z)} \nabla_h \tilde{y}^h (x) =  \Delta^{(z)} \nabla \tilde{w} (h^{-1} x_1, x_2)$, 
by setting $S = S(h) = \{ x \in h^{-1} \Omega_h : (h x_1, x_2) \in U \}$ we get 
\begin{align}\label{eq:nabla-ty-sec-part} 
  \| h^{-1} e \big( \Delta^{(z)} \tilde{R}^T \nabla_h \tilde{y} \big) \|_{L^p(U)}^p 
  = h \| h^{-1} e \big( \tilde{R}^T(h\,\cdot) \Delta^{(z)} \nabla \tilde{w} \big) \|_{L^p(S)}^p 
  \le \hat{C} 
\end{align} 
by Lemma \ref{lemma:tw-diff-quot}(iii) which in particular shows that $h^{-1} \big( \Delta^{(z)} \tilde{R}^T \nabla_h \tilde{y} \big)_{11}$ is bounded in $L^p$. In order to identify its weak limit we first observe that Lemma \ref{lemma:tw-diff-quot}(i) and (ii) also gives the estimates 
\begin{align}\label{eq:ty-elast-est} 
  \| h^{-1} \Delta^{(z)} \tilde{R}^T \tilde{y}^h - \mathbf{e}_2 \|_{L^2(U)}^2 
  = h \| \Delta^{(z)} \tilde{w} - \tilde{R}(h\,\cdot) \mathbf{e}_2 \|_{L^2(S)}^2 
  \le \hat{C} h^2  
\end{align} 
as well as 
\begin{align}\label{eq:ty-crack-est} 
  {\cal H}^1(J_{h^{-1} \Delta^{(z)} \tilde{R}^T \tilde{y}} \cap U)
  = {\cal H}^1(J_{\Delta^{(z)} \tilde{y}} \cap U) 
  \le {\cal H}^1(J_{\Delta^{(z)} \tilde{w}} \cap S) 
  \le C, 
\end{align} 
where we have used that $\tilde{R}$ does not depend on $x_2$. The rescaled absolutely continuous part of the derivative of $h^{-1} \Delta^{(z)} \tilde{R}^T \tilde{y}$ is given by 
\begin{align*} 
  \nabla_h \big( h^{-1} \tilde{R}^T \Delta^{(z)} \tilde{y} \big) 
  &= \big( h^{-1} (\tilde{R}^T)' \Delta^{(z)} \tilde{y} \mid 0 \big) + h^{-1} \tilde{R}^T \nabla_h \Delta^{(z)} \tilde{y} 
   =: A^h_1 + A^h_2, 
\end{align*}
where the symmetric part of the second summand on the right hand side has been estimated in \eqref{eq:nabla-ty-sec-part} while for the first one Lemma \ref{lemma:tw-diff-quot}(i) and \eqref{eq:R-conv-sw} give  
\begin{align}\label{eq:A1-est} 
  \| A^h_1 - \big( (\tilde{R}^T)' \tilde{R} \, \mathbf{e}_2 \mid 0 \big) \|_{L^1(U)} 
  &= \| h^{-1} (\tilde{R}^T)' \Delta^{(z)} \tilde{y} - (\tilde{R}^T)' \tilde{R} \, \mathbf{e}_2 \|_{L^1(U)} \\ 
  &\le \| (\tilde{R}^T)' \|_{L^2(U)} h \| \Delta^{(z)} \tilde{w} - \tilde{R}(h\,\cdot) \mathbf{e}_2 \|_{L^2(S)} 
  \le \hat{C} h^{\frac{3}{2}}.  
\end{align} 
Here by \eqref{eq:R-conv-sw} the product $(\tilde{R}^T)' \tilde{R}$ converges weakly in $L^1$ to 
\begin{align}\label{eq:Rstrich-R-conv}
  \big( \partial_{11} y \mid \partial_1 (\partial_1 y)^{\perp})^T  (\partial_1 y \mid (\partial_1 y)^{\perp} \big) 
  = \begin{pmatrix} 
        0 & \partial_{11} y \cdot (\partial_1 y)^{\perp} \\ 
        - \partial_{11} y \cdot (\partial_1 y)^{\perp} & 0 
     \end{pmatrix}, 
\end{align} 
where we have used that $\partial_1 (\partial_1 y)^{\perp} \cdot (\partial_1 y)^{\perp} = \partial_{11} y \cdot \partial_1 y = \frac{1}{2} \partial_1 | \partial_1 y |^2 = 0$. 

Now let $H = (H_{ij}) = \mathbf{e}_1 \otimes \mathbf{e}_1 + h \, \mathbf{e}_2 \otimes \mathbf{e}_2$ be the diagonal $2 \times 2$ matrix with $H_{11} = 1$ and $H_{22} = h$ and consider the auxiliary function $f^h \in SBV(U; \R^2)$ defined by 
$$ f^h(x) 
   = H h^{-1} \Delta^{(z)} \tilde{R}^T(x_1) \tilde{y} (x). $$ 
By \eqref{eq:ty-elast-est} and \eqref{eq:ty-crack-est} we then have 
\begin{align*} 
  f^h \to 0 ~~\mbox{strongly in } L^1 
  \quad\mbox{and}\quad
  {\cal H}^1(J_{f^h}) \le C. 
\end{align*} 
Moreover, since 
\begin{align*} 
  \nabla f^h 
  = H \nabla_h \big( h^{-1} \Delta^{(z)} \tilde{R}^T \tilde{y} \big) H 
  = H A_1^h H + H A_2^h H,   
\end{align*}
where $H A_1^h H \weakly \partial_{11} y \cdot (\partial_1 y)^{\perp} \mathbf{e}_1 \otimes \mathbf{e}_1$ weakly in $L^1$ by \eqref{eq:A1-est} and \eqref{eq:Rstrich-R-conv} and $e ( H A_2^h H )$ is bounded in $L^p$ by \eqref{eq:nabla-ty-sec-part}. We may thus apply the SBV closure result stated in Theorem \ref{theo:SBD-closure} (cf.\ also Remark \ref{rem:DPdlVP}) in order to conclude that $e ( H A_1^h H + H A_2^h H ) \weakly e( \nabla 0 ) = 0$ and thus 
$$ e ( H A_2^h H ) 
   \weakly - \partial_{11} y \cdot (\partial_1 y)^{\perp} \mathbf{e}_1 \otimes \mathbf{e}_1 $$ 
weakly in $L^1$. In particular it follows that 
$$ \big( h^{-1} \Delta^{(z)} \tilde{R}^T \nabla_h \tilde{y} \big)_{11} 
   = (H A_2^h H)_{11} 
   \weakly - \partial_{11} y \cdot (\partial_1 y)^{\perp} $$
weakly in $L^1$ and so, being bounded in $L^p$, also weakly in $L^p$ as claimed. 
\end{proof}

\subsection{The Gamma lim\,inf inequality}

Thanks to the preparations in the previous sections, it is now possible to follow the strategy for elastic plates devised in \cite{FJM:02} so as to estimate the elastic part in the proof of the $\liminf$ inequality in Theorem \ref{theo:Gamma-convergence}. Some extra care, however, is needed as the bounds on $F^{(1)}$ and $G^{(1)}$ in Lemma \ref{lemma:F1-global-est} and \ref{lemma:G-ests}, respectively, and the weak convergence in Lemma \ref{lemma:delta-z-tw}  only hold with $p < 2$ and also the contributions from regions where rectangles in ${\cal G}$ overlap have to be estimated.

\begin{proof}[Proof of Theorem \ref{theo:Gamma-convergence}(i).] 
Let $(y^h)$ be a sequence in $SBV(\Omega; \R^2)$ with $y^h \to y$ in $L^1(\Omega)$. Without loss of generality we may suppose that $\liminf_{h \to 0} I^h(y^h) < \infty$ and, passing to a subsequence (not relabeled) which realizes this $\liminf$ as its limit, assume that $I^h(y^h) \le C$ for some constant $C > 0$ so that $y^h \in {\cal A}^h$ and $y \in {\cal A}$ and $\nabla_h y^h \to \big( \partial_1 y \mid (\partial_1 y)^{\perp} \big)$ strongly in $L^2$ by Theorem \ref{theo:compactness} . This also justifies our passing to further subsequences in the sequel. Rescaling $W$, if necessary, we can in addition assume that $\beta = 1$. We need to show that 
\begin{align*}
  \liminf_{h \to 0} \frac{1}{h} \int_{h^{-1} \Omega_h} W(\nabla w^h) + {\cal H}^1 (J_{w^h})
  &\ge \frac{\alpha}{24} \int_{\Omega} |y'' \cdot (y')^{\perp}|^2 \, dx + {\cal H}^1(J_{y} \cup J_{y'}). 
\end{align*}

With $\bar{u}_a$ as in \eqref{eq:bu-def} let $\chi^h$ be the characteristic function of the set 
$$ \big\{ x \in \Omega : |\nabla \bar{u}_a(h^{-1}x_1, x_2)| \le h^{\frac{4}{5}} \mbox{ for all } a \mbox{ with } (h^{-1}x_1, x_2) \in Q_{a,\rho} \mbox{ and } Q_a \in {\cal G} \big\}. $$ 
Note that by \eqref{eq:bu-ests} and \eqref{eq:sum-epsa-bound} 
\begin{align}\label{eq:chi-h-est}
\begin{split}
  | \{ x \in \Omega : \chi^h(x) \ne 1 \} | 
  &\le h \sum_{a \in {\cal G}} | \{ x \in Q_{a,\rho} : |\nabla \bar{u}_a(x_1, x_2)|^2 > h^{\frac{8}{5}} \} | \\ 
  &\le h \sum_{a \in {\cal G}} h^{-\frac{8}{5}} \| \nabla \bar{u}_a \|_{L^2(Q_{a,\rho})}^2 
   \le \hat{C} h^{-\frac{3}{5}} \sum_{a \in {\cal G}} \eps_a^{\frac{9}{10}} \\ 
  &\le \hat{C} h^{-\frac{3}{5} - \frac{1}{10}} \bigg( \sum_{a \in {\cal G}} \eps_a \bigg)^{\frac{9}{10}} 
   \le \hat{C} h^{\frac{1}{5}}. 
\end{split}
\end{align}

Using Corollary \ref{cor:jump-in-B} we now pass to a subsequence (not relabeled) such that $h I^{\cal B} \to J = \{t_1, \ldots, t_m\} \supset J_y \cup J_{y'}$ for suitable $0 \le t_1 < t_2 < \ldots < t_m \le L$. We fix $0 < \delta < \frac{1}{4} \min\{|t_i - t_{i-1}| : i = 2, \ldots, m\}$ and denote by $J_{\delta} = \{t \in (0, L) : \dist(t, J) < \delta\}$ the $\delta$-neighborhood of $J$. 

We proceed to estimate the elastic energy away from $J$. Suppose $i \in \{1, \ldots, m+1\}$ is such that $(t_{i-1} + \delta, t_i - \delta)$ is non empty, where $t_0 = 0$ and $t_{m+1} = L$. Set $\lfloor h^{-1}(t_{i-1}+\delta) \rfloor = a$ and $\lfloor h^{-1}(t_{i}-\delta) \rfloor = b$. For sufficiently small $h$ we then have $Q_{a-1}, \ldots, Q_{b+1} \in {\cal G}$ and the projection of $Q_{a-1} \cup \ldots \cup Q_{b+1}$ onto the $x_1$-axis contains $h^{-1}(t_{i-1} + \delta, t_i - \delta)$ and is itself contained in $h^{-1}(t_{i-1} + \frac{2\delta}{3}, t_i - \frac{2\delta}{3})$. For such an $i$ one has 
\begin{align}\label{eq:W-J-weg}
\begin{split}
  &\sum_{j=a}^{b} \int_{Q_j} W(\nabla w^h) + h {\cal H}^1 (J_{w^h} \cap Q_j) \\ 
  &~~\ge \frac{h}{2} \int_{h(a-1)}^{hb} \int_{-\frac{1}{2} + \bar{c} \rho}^{\frac{1}{2} - \bar{c} \rho} {\cal Q} \big( \chi^h G^{(2)} \big) 
    - \hat{C} h^{\frac{11}{10}} - C \rho h
\end{split}
\end{align}
by Lemma \ref{lemma:en-cell-string} and construction of $\chi^h$ with $G^{(2)}$ as in \eqref{eq:F-def}. 

The crack energy on the other hand is estimated by noting that for each $t_i$, $i = 1, \ldots, m$, and sufficiently small $h$ there is a $Q_{a_i} \in {\cal B}$ such that the projection of $Q_{a_i-1} \cup Q_{a_i} \cup Q_{a_i+1}$ onto the $x_1$-axis is contained in $h^{-1}(t_i - \frac{\delta}{3}, t_i + \frac{\delta}{3})$. For each $i$ one then has  
\begin{align}\label{eq:W-J-da}
  \sum_{j=a_i-1}^{a_i+1} \int_{Q_j} W(\nabla w^h) + h {\cal H}^1 (J_{w^h} \cap Q_j) 
  \ge h \lambda 
\end{align}
by Lemma \ref{lemma:B-bounds}(i), for small $h$. Noting that all the rectangles considered in \eqref{eq:W-J-weg} and \eqref{eq:W-J-da} are distinct, we may sum these estimates over all $i$ to arrive at  
\begin{align*}
  &\sum_{a=1}^N \int_{Q_a} W(\nabla w^h) + h {\cal H}^1 (J_{w^h} \cap Q_a) \\ 
  &~~\ge \frac{h}{2} \int_{(0,L) \setminus J_{\delta}} \int_{-\frac{1}{2} + \bar{c} \rho}^{\frac{1}{2} - \bar{c} \rho} {\cal Q} \big( \chi^h G^{(2)} \big) 
    + h m \lambda - \hat{C} h^{\frac{11}{10}} - C \rho h. 
\end{align*}

By Lemma \ref{lemma:G-ests} there is a symmetric $G^{(2)}_0 \in L^2 \big( ((0, L) \setminus J_{\delta}) \times (-\frac{1}{2} + \bar{c} \rho, \frac{1}{2} - \bar{c} \rho); \R^{2 \times 2}_{\rm sym} \big)$ such that (passing to a subsequence) $G^{(2)} \weakly G^{(2)}_0$ in $L^2$. Since $\chi_h \to 1$ boundedly in measure due to \eqref{eq:chi-h-est} and thus still $\chi_h G^{(2)} \weakly G^{(2)}_0$ in $L^2$, by lower semicontinuity and \eqref{eq:Q-def} we now obtain 
\begin{align*}
  &\liminf_{h \to 0} \frac{1}{h} \sum_{a=1}^N \int_{Q_a} W(\nabla w^h) + {\cal H}^1 (J_{w^h} \cap Q_a) \\ 
  &~~\ge \frac{1}{2} \int_{(0,L) \setminus J_{\delta}} \int_{-\frac{1}{2} + \bar{c} \rho}^{\frac{1}{2} - \bar{c} \rho} {\cal Q} \big( G^{(2)}_0 \big) 
    + m \lambda - C \rho \\ 
  &~~\ge \frac{\alpha}{2} \int_{(0,L) \setminus J_{\delta}} \int_{-\frac{1}{2} + \bar{c} \rho}^{\frac{1}{2} - \bar{c} \rho} |g^{(2)}_{11}|^2 
    + m \lambda - C \rho, 
\end{align*}
where $g^{(2)}_{11}$ is the upper left entry in $G^{(2)}_0$. 

Fix $1 < p < 2$ and note that, upon passing to further subsequences, Lemma \ref{lemma:G-ests} also yields $G^{(1)}_0 = (g^{(1)}_{ij})$ and $G_0 = (g_{ij})$ in $L^p \big( ((0, L) \setminus J_{\delta}) \times (-\frac{1}{2} + \bar{c} \rho, \frac{1}{2} - \bar{c} \rho); \R^{2 \times 2}_{\rm sym} \big)$ such that $G^{(1)} \weakly G^{(1)}_0$ and $G \weakly G_0$ in $L^p$, where $G^{(2)}_0 = G_0 - G^{(1)}_0$ and $\| G^{(1)}_0 \|_{L^p}^p \le C \rho^{\frac{p}{2}}$. Defining the convex function $f_{\delta} : \R \to [0,\infty)$ by 
$$ f_{\delta}(t) 
   = \begin{cases} 
        t^2 & \mbox{for } |t| \le \delta^{-1}, \\ 
        2 \delta^{-1} |t| - \delta^{-2} & \mbox{for } |t| \ge \delta^{-1}, 
     \end{cases} $$ 
which obviously satisfies $(t+s)^2 \ge f_{\delta}(t+s) \ge f_{\delta}(t) - 2 \delta^{-1} |s|$ for all $t,s$, we thus obtain 
\begin{align*}
  &\liminf_{h \to 0} \frac{1}{h} \sum_{a=1}^N \int_{Q_a} W(\nabla w^h) + {\cal H}^1 (J_{w^h} \cap Q_a) \\ 
  &~~ \ge \frac{\alpha}{2} \int_{(0,L) \setminus J_{\delta}} \int_{-\frac{1}{2} + \bar{c} \rho}^{\frac{1}{2} - \bar{c} \rho} 
      f_{\delta} (g_{11}) - 2 \delta^{-1} |g^{(1)}_{11}| \, dx 
      + m \lambda - C \rho \\ 
  &~~ \ge \frac{\alpha}{2} \int_{(0,L) \setminus J_{\delta}} \int_{-\frac{1}{2} + \bar{c} \rho}^{\frac{1}{2} - \bar{c} \rho} 
      f_{\delta} (g_{11})  
      + m \lambda - C \delta^{-1} \rho^{\frac{1}{2}}. 
\end{align*}

To further identify $g_{11}$ we note that for any $U \subset\subset ((0, L) \setminus J_{\delta}) \times (-\frac{1}{2} + \bar{c} \rho, \frac{1}{2} - \bar{c} \rho)$ and $z \in \R$ with $|z| < \dist \big( \partial U, ((0, L) \setminus J_{\delta}) \times (-\frac{1}{2} + \bar{c} \rho, \frac{1}{2} - \bar{c} \rho) \big)$ by definition of $G$ and Lemma \ref{lemma:delta-z-tw} we have 
\begin{align*}
  \Delta^{(z)} g_{11} 
  = \mbox{w-}\lim \big( h^{-1} \Delta^{(z)} \tilde{R}^T \nabla_h y^h \big)_{11} 
  = \partial_{11} y \cdot (\partial_1 y)^{\perp} 
\end{align*}
in $L^p(U)$ so that 
$$ g_{11}(x) = g_{11}(x_1, 0) + x_2 \partial_{11} y(x_1) \cdot (\partial_1 y)^{\perp}(x_1) $$
on $((0, L) \setminus J_{\delta}) \times (-\frac{1}{2} + \bar{c} \rho, \frac{1}{2} - \bar{c} \rho)$. 

Using that $f_{\delta}(t + s) \ge f_{\delta}(t) + f_{\delta}'(t) s$ for all $t,s$ and that $f_{\delta}'$ is odd we now get 
\begin{align*}
  &\liminf_{h \to 0} \frac{1}{h} \sum_{a=1}^N \int_{Q_a} W(\nabla w^h) + {\cal H}^1 (J_{w^h} \cap Q_a) \\ 
  &~~\ge \frac{\alpha}{2} \int_{(0,L) \setminus J_{\delta}} \int_{-\frac{1}{2} + \bar{c} \rho}^{\frac{1}{2} - \bar{c} \rho} 
    f_{\delta} \big( x_2 \, y'' \cdot (y')^{\perp} \big) \\ 
  &\qquad\qquad\qquad\qquad\qquad + f_{\delta}'\big( x_2 \, y'' \cdot (y')^{\perp} \big) g_{11}(x_1, 0)  \, dx 
    + m \lambda - C \delta^{-1} \rho^{\frac{1}{2}} \\ 
   &~~= \frac{\alpha}{2} \int_{(0,L) \setminus J_{\delta}} \int_{-\frac{1}{2} + \bar{c} \rho}^{\frac{1}{2} - \bar{c} \rho} 
    f_{\delta} \big( x_2 \, y'' \cdot (y')^{\perp} \big) \, dx 
    + m \lambda - C \delta^{-1} \rho^{\frac{1}{2}}. 
\end{align*}
Now sending $\rho \to 0$, $\lambda \to 1$ and $\delta \to 0$ and carrying out the $x_2$ integration we get 
\begin{align*}
  &\liminf_{h \to 0} \frac{1}{h} \sum_{a=1}^N \int_{Q_a} W(\nabla w^h) + {\cal H}^1 (J_{w^h} \cap Q_a) \\ 
  &~~\ge \frac{\alpha}{2} \int_0^L \int_{-\frac{1}{2}}^{\frac{1}{2}}  | x_2 \, y'' \cdot (y')^{\perp}|^2 + {\cal H}^1(J_{y} \cup J_{y'}) \\ 
  &~~= \frac{\alpha}{24} \int_0^L |y'' \cdot (y')^{\perp}|^2 + {\cal H}^1(J_{y} \cup J_{y'}). 
\end{align*}
by monotone convergence and the fact that $m \ge {\cal H}^1(J_{y} \cup J_{y'})$. 

Finally, for each $k = 0, \ldots, n-1$ we repeat the above analysis with the shifted rectangles $Q^{(k)}_{1} = Q_{1} + \frac{k}{n}, \ldots, Q^{(k)}_{N} = Q_{N} + \frac{k}{n}$ to obtain 
\begin{align*}
  &\liminf_{h \to 0} \frac{1}{h} \sum_{a=1}^N \int_{Q^{(k)}_a} W(\nabla w^h) + {\cal H}^1 (J_{w^h} \cap Q^{(k)}_a) \\ 
  &~~\ge \frac{\alpha}{24} \int_0^L |y'' \cdot (y')^{\perp}|^2 + {\cal H}^1(J_{y} \cup J_{y'}). 
\end{align*}
for any such $k$. Then summing over $k$ yields 
\begin{align*}
  &(n+1) \liminf_{h \to 0} \frac{1}{h} \int_{h^{-1} \Omega_h} W(\nabla w^h) + {\cal H}^1 (J_{w^h}) \\ 
  &~~\ge \liminf_{h \to 0} \sum_{k = 0}^{n-1} \bigg( \frac{1}{h} \sum_{a=1}^{N} \int_{Q^{(k)}_a} W(\nabla w^h) + {\cal H}^1 (J_{w^h} \cap Q^{(k)}_a) \bigg) \\ 
  &~~\ge n \bigg( \frac{\alpha}{24} \int_0^L |y'' \cdot (y')^{\perp}|^2 + {\cal H}^1(J_{y} \cup J_{y'}) \bigg). 
\end{align*}
Dividing by $n$ and sending $n \to \infty$ we indeed get 
\begin{align*}
  \liminf_{h \to 0} \frac{1}{h} \int_{h^{-1} \Omega_h} W(\nabla w^h) + {\cal H}^1 (J_{w^h})
  &\ge \frac{\alpha}{24} \int_0^L |y'' \cdot (y')^{\perp}|^2 + {\cal H}^1(J_{y} \cup J_{y'}). 
\end{align*}
\end{proof}

\subsection{Recovery sequences, body forces and boundary values}

We will now complete the proof of Theorem \ref{theo:Gamma-convergence} by establishing the existence of recovery sequences and also prove Corollaries \ref{cor:Gamma-compactness-forces} and \ref{cor:Gamma-compactness-forces-bv}. Thanks to our limiting functional being one-dimensional, the arguments are rather straightforward. We include them for the sake of completeness.

\begin{proof}[Proof of Theorem \ref{theo:Gamma-convergence}(ii).] 
Recall that, by assumption, $M \gg \max\{1, L\}$. Without loss of generality we assume $y \in {\cal A}$ with, say, $J_{y} \cup J_{y'} = \{ t_1, \ldots, t_m \}$. Also set $t_0 = 0$ and $t_{m+1} = L$. Suppose first that $y|_{(t_{i-1}, t_i)} \in C^{\infty}([t_{i-1}, t_i])$ for all $i \in \{1, \ldots, m+1\}$ and that $\| y \|_{L^{\infty}(0,L)} < M$. 

Since ${\cal Q}$ is positive definite on symmetric matrices, for each $\lambda \in \R$ there is a unique vector $\gamma(\lambda) \in \R^2$ such that 
$$ {\cal Q} \big( \lambda \mathbf{e}_1 \mid \gamma(\lambda) \big) 
   = \min_{\gamma \in \R^2} {\cal Q} \big( \lambda \mathbf{e}_1 \mid \gamma \big) 
   = \alpha \lambda^2, $$ 
which moreover depends linearly on $\lambda$, cf.\ \eqref{eq:Q-def}. Set $R = \big( y' \mid {y'}^{\perp} \big)$, $\kappa = y'' \cdot {y'}^{\perp}$, $d = R \gamma(-\kappa)$ and define $y^h \in SBV(\Omega; \R^2)$ by 
$$ y^h(x) 
   = y(x_1) + h x_2 {y'}^{\perp}(x_1) + \frac{h^2 x_2^2}{2} d(x_1) $$ 
so that $J_{y^h} = \{ t_1, \ldots, t_m \} \times (\frac{1}{2}, \frac{1}{2})$ for small $h$. We will now show that $(y^h)$ is indeed a recovery sequence for $y$. 

First note that since $|d| = |\gamma(- \kappa)| \le C$, for $h$ small we have 
\begin{align}\label{eq:yh-M-est}
  |y^h| 
  \le \| y \|_{L^{\infty}} + h + C h^2 
  \le M 
  \quad\mbox{and}\quad
  y^h \to y \mbox{ in } L^1. 
\end{align} 
With $\nabla_h y^h = \big( y' \mid {y'}^{\perp} \big) + h x_2 \big( {y''}^{\perp} \mid d \big) + \frac{h^2 x_2^2}{2} \big( d' \mid 0 \big)$ we then compute, using that ${y'}^{\perp} \cdot {y''}^{\perp} = \frac{1}{2} {(|y'|^2)'} = 0$, 
\begin{align*}
   R^T \nabla_h y^h 
  &= \Id + h x_2 \big( - \kappa \mathbf{e}_1 \mid \gamma(-\kappa) \big) + \frac{h^2 x_2^2}{2} \big( R^T d' \mid 0 \big), 
\end{align*}
where $|x_2 \big( - \kappa^h \mathbf{e}_1 \mid \gamma(-\kappa) \big)| + | \frac{x_2^2}{2} \big( R^T d' \mid 0 \big)| \le C$. In particular, also $|\nabla_h y^h| \le \sqrt{2} + C h\le M$ and thus $y^h \in {\cal A}^h$ by \eqref{eq:yh-M-est}. 
 
Taylor expanding the frame indifferent $W$ as in \eqref{eq:Taylor-exp} now yields 
\begin{align*}
  W( \nabla_h y^h ) 
  &= \frac{h^2 x_2^2}{2} {\cal Q} \big( \big( - \kappa \mathbf{e}_1 \mid \gamma(-\kappa) \big) \big) + B^h 
\end{align*}
with $|B^h(x)| \le C h^3$. It follows that 
\begin{align*}
  &\lim_{h \to 0} h^{-2} \int_{\Omega} W(\nabla_h y^h) \, dx 
    + \beta \int_{\Omega \cap J_{y^h}} |(\nu_1(y^h), h^{-1} \nu_2(y^h))| \, d{\cal H}^1 \\ 
  &~~= \lim_{h \to 0} \frac{1}{2} \int_0^L \int_{-\frac{1}{2}}^{\frac{1}{2}}  x_2^2 {\cal Q} \big( \big(- \kappa \mathbf{e}_1 \mid \gamma(-\kappa) \big) \big) + \beta m \\ 
  &~~= \lim_{h \to 0} \frac{\alpha}{24} \int_0^L  |\kappa(x_1)|^2 \, dx_1 + \beta \# (J_{y} \cup J_{y'}) 
   = I^0(y).  
\end{align*}

For general $y \in {\cal A}$ with, say, $J_{y} \cup J_{y'} = \{ t_1, \ldots, t_m \}$, $t_0 = 0$ and $t_{m+1} = L$, it now suffices to observe that there is a sequence $y_k \in {\cal A}$ with $J_{y_k} \cup J_{y_k'} = \{ t_1, \ldots, t_m \}$ such that $y_k|_{(t_{i-1}, t_i)} \in C^{\infty}([t_{i-1}, t_i])$ and $y_k|_{(t_{i-1}, t_i)} \to y|_{(t_{i-1}, t_i)}$ in $W^{2,2}(t_{i-1}, t_i)$ for every $i \in \{1, \ldots, m+1\}$ and in addition $\| y_k \|_{L^{\infty}(0,L)} < M$. Since then $y_k'' \cdot y_k'^{\perp} \to y'' \cdot y'^{\perp}$ in $L^2$, a recovery sequence for $y$ can be obtained by choosing a suitable diagonal sequence. 
\end{proof}

\begin{proof}[Proof of Corollary \ref{cor:Gamma-compactness-forces}.] 
The compactness statement is immediate from Theorem \ref{theo:compactness}, Remark \ref{rem:main-res}.\ref{rem:main-res-better-top} and the estimate 
$$ |J^h(y) - I^h(y)| 
   \le M \| h^{-2} f^h \|_{L^1} 
   \le C $$ 
for all $y \in {\cal A}^h$. 

The Gamma-convergence statement follows by noting that if $J^h(y^h) \le C$ and $y^h \to y$ in $L^1$, then in fact $y^h \to y$ boundedly in measure due to the uniform $L^{\infty}$-bound on $y^h$ and so 
$$ h^{-2} \int_{\Omega} y^h(x) \cdot f^h(x) \, dx 
   \to \int_{\Omega} y(x) \cdot f(x) \, dx 
   = \int_0^L \bar{y}(x_1) \cdot \bar{f}(x_1) \, dx_1, $$ 
where $y \in {\cal A}$ with $y(x) = \bar{y}(x_1)$ for a.e.\ $x$. 
\end{proof}

\begin{proof}[Proof of Corollary \ref{cor:Gamma-compactness-forces-bv}.] 
The compactness statement and, as a consequence, the $\liminf$ inequality are straightforward from Corollary \ref{cor:Gamma-compactness-forces}. It only remains to show that, for $y \in  {\cal A_{\rm bv}}$ a recovery sequence can be chosen which not only lies in ${\cal A}^h$ but also in ${\cal A}^h_{\rm bv}$. If $y$ is piecewise smooth, then the recovery sequence $y^h$ constructed in the proof of Theorem \ref{theo:Gamma-convergence}(ii) does indeed attain prescribed boundary values on $((- \eta, 0) \cup (L, L + \eta)) \times (\frac{1}{2}, \frac{1}{2})$. It now suffices to observe that for general $y \in {\cal A}_{\rm bv}$ with, say, $J_{y} \cup J_{y'} = \{ t_1, \ldots, t_m \}$, $t_0 = 0$ and $t_{m+1} = L$ the approximating $y_k \in {\cal A}$ with $J_{y_k} \cup J_{y_k'} = \{ t_1, \ldots, t_m \}$, $y_k|_{(t_{i-1}, t_i)} \in C^{\infty}([t_{i-1}, t_i])$ and $y_k|_{(t_{i-1}, t_i)} \to y|_{(t_{i-1}, t_i)}$ in $W^{2,2}(t_{i-1}, t_i)$ for every $i \in \{1, \ldots, m+1\}$ and $\| y_k \|_{L^{\infty}(0,L)} < M$ can be chosen in such a way that $y_k = y$ on $((- \eta, 0) \cup (L, L + \eta)$ for all $k$. 
\end{proof}

\section*{Acknowledgment}

I am grateful to Manuel Friedrich for our interesting discussions and his helpful advice related to the quantitative piecewise geometric rigidity result stated in Theorem \ref{theo:quant-rig}. 


 \typeout{References}


\begin{thebibliography}{10}

\bibitem{Ambrosio:1989} 
{\sc L~ Ambrosio}.
\newblock {\em A compactness theorem for a new class of functions of bounded variation}. 
\newblock Boll.\ Un.\ Mat.\ Ital.\ B 
\newblock {\bf 3} (1989), 857--881.


\bibitem{Ambrosio-Coscia-DalMaso:1997} 
{\sc L~ Ambrosio, A.~Coscia, G.~Dal Maso}.
\newblock {\em Fine properties of functions with bounded deformation}. 
\newblock Arch.\ Ration.\ Mech.\ Anal.\
\newblock {\bf 139} (1997), 201--238.

\bibitem{Ambrosio-Fusco-Pallara:2000} 
{\sc L.~Ambrosio, N.~Fusco, D.~Pallara}.
\newblock {\em Functions of bounded variation and free discontinuity problems}. 
\newblock Oxford University Press, Oxford 2000. 

\bibitem{anzbalper}
{\sc G.~Anzellotti, S.~Baldo, D.~Percivale}.
\newblock {\em Dimension reduction in variational problems, asymptotic development in $\Gamma$-convergence and thin structures in elasticity}.
\newblock Asympt.\ Anal.\ 
\newblock {\bf 9} (1994), 61--100.

\bibitem{Babadjian}
{\sc J.-F.~Babadjian}.  
\newblock {\em Quasistatic evolution of a brittle thin film}. 
\newblock Calc.\ Var.\ Partial Differential Equations 
\newblock {\bf 26} (2006), 69--118. 

\bibitem{Bellettini-Coscia-DalMaso:1998}
{\sc G.~Bellettini, A.~Coscia, G.~Dal Maso}.
\newblock {\em  Compactness and lower semicontinuity properties in $SBD(\Omega)$}.
\newblock  Math.\ Z.\
\newblock {\bf 228} (1998), 337--351.

\bibitem{Bourdin-Francfort-Marigo:2008}
{\sc B.~Bourdin, G.~A.~Francfort, J.~J.~Marigo}. 
\newblock {\em The variational approach to fracture}.
\newblock J.\ Elasticity\ 
\newblock {\bf 91} (2008), 5--148. 


\bibitem{BraidesFonseca:01} 
{\sc A.~Braides, I.~Fonseca}. 
\newblock {\em Brittle thin films}. 
\newblock Appl.\ Math.\ Optim.\ 
\newblock {\bf 44} (2001), 299--323. 

\bibitem{Braides-Lew-Ortiz:06}
{\sc A.~Braides, A.~Lew, M.~Ortiz}.
\newblock {\em Effective cohesive behavior of layers of interatomic planes}.
\newblock Arch.\ Ration.\ Mech.\ Anal.\ 
\newblock {\bf 180} (2006), 151--182.

\bibitem{Chambolle-Giacomini-Ponsiglione:2007}
{\sc A.~Chambolle, A.~Giacomini, M.~Ponsiglione}. 
\newblock {\em Piecewise rigidity}.
\newblock J.\ Funct.\ Anal.\ 
\newblock {\bf 244} (2007), 134--153. 

\bibitem{Ciarlet-ii}
{\sc P.~G.~Ciarlet}.
\newblock {\em Mathematical elasticity Vol.\ II: Theory of plates}.
\newblock North-Holland Publishing Co., Amsterdam, 1997.

\bibitem{Ciarlet-iii}
{\sc P.~G.~Ciarlet}.
\newblock {\em Mathematical elasticity Vol.\ III: Theory of shells}. 
\newblock North-Holland Publishing Co., Amsterdam, 2000. 

\bibitem{ContiDolzmann:09}
{\sc S.~Conti, G.~Dolzmann}. 
\newblock {\em $\Gamma$-convergence for incompressible elastic plates}. 
\newblock Calc.\ Var.\ Partial Differential Equations 
\newblock {\bf 34} (2009), 531--551.

\bibitem{dalmaso}
{\sc G.~Dal Maso}.
\newblock {\em An introduction to $\Gamma$-convergence}.
\newblock Birkh{\"a}user, Boston $\cdot$ Basel $\cdot$ Berlin 1993.

\bibitem{DalMaso:13}
{\sc G.~Dal Maso}.
\newblock {\em Generalised functions of bounded deformation}.
\newblock J.\ Eur.\ Math.\ Soc.\
\newblock {\bf 15} (2013), 1943--1997.


\bibitem{Euler}
{\sc L.~Euler}.
\newblock {\em Methodus Inveniendi Lineas Curvas, Additamentum I: De Curvis Elasticis} (1744).
\newblock In: Opera Omnia Ser.\ Prima Vol.\ XXIV, pp.\ 231-297, Orell F\"ussli, Bern 1952.


\bibitem{Francfort-Marigo:1998}
{\sc G.~A.~Francfort, J.~J.~Marigo}. 
\newblock {\em Revisiting brittle fracture as an energy minimization problem}.
\newblock J.\ Mech.\ Phys.\ Solids 
\newblock {\bf 46} (1998), 1319--1342. 

\bibitem{Friedrich:15a}
{\sc M.~Friedrich}.
\newblock {\em A Korn-Poincar\'e-type inequality for special functions of bounded deformation}. 
\newblock Preprint,\ 2015. 
\newblock http://arxiv.org/abs/1503.06755

\bibitem{Friedrich:15b}
{\sc M.~Friedrich}.
\newblock {\em A derivation of linearized Griffith energies from nonlinear models}. 
\newblock Preprint,\ 2015. 
\newblock http://arxiv.org/abs/1504.01671 

\bibitem{FriedrichSchmidt:14}
{\sc M.~Friedrich, B.~Schmidt}.
\newblock {\em An atomistic-to-continuum analysis of crystal cleavage in a two-dimensional model problem}. 
\newblock J.\ Nonlin.\ Sci.\ 
\newblock {\bf 24} (2014), 145--183.

\bibitem{FriedrichSchmidt:15a}
{\sc M.~Friedrich, B.~Schmidt}.
\newblock {\em An analysis of crystal cleavage in the passage from atomistic models to continuum theory}. 
\newblock Arch.\ Rational\ Mech.\ Anal.,\  
\newblock {\bf 217} (2015), 263--308. 


\bibitem{FriedrichSchmidt:15c}
{\sc M.~Friedrich, B.~Schmidt}.
\newblock {\em On a discrete-to-continuum convergence result for a two dimensional brittle material in the
small displacement regime}. 
\newblock Netw.\ Heterog.\ Media 
\newblock {\bf 10} (2015), 321--342. 

\bibitem{FriedrichSchmidt:15b}
{\sc M.~Friedrich, B.~Schmidt}.
\newblock {\em A quantitative geometric rigidity result in SBD}. 
\newblock Preprint,\ 2015. 
\newblock http://arxiv.org/abs/1503.06821


\bibitem{FJM:02}
{\sc G.~Friesecke, R.~D.~James, S.~M{\"u}ller}.
\newblock {\em A theorem on geometric rigidity and the derivation of nonlinear plate theory from three-dimensional elasticity}.
\newblock Comm.\ Pure Appl.\ Math.\ 
\newblock {\bf 55} (2002), 1461--1506.


\bibitem{FJM:06}
{\sc G.~Friesecke, R.~D.~James, S.~M{\"u}ller}.
\newblock {\em A hierarchy of plate models derived from nonlinear elasticity by Gamma-convergence}.
\newblock Arch.\ Rational Mech.\ Anal.\
\newblock {\bf 180} (2006), 183--236.

\bibitem{fjmm}
{\sc G.~Friesecke, R.~D.~James, M.~G.~Mora, S.~M{\"u}ller}.
\newblock {\em Derivation of nonlinear bending theory for shells from three-dimensional nonlinear elasticity by Gamma-convergence}.
\newblock C.\ R.\ Acad.\ Sci.\ Paris  
\newblock {\bf 336} (2003), 697--702.


\bibitem{fjmmii}
{\sc G.~Friesecke, R.~D.~James, M.~G.~Mora, S.~M{\"u}ller}.
\newblock {\em Derivation of the nonlinear bending-torsion theory for inextensible rods by $\Gamma$-convergence}.
\newblock Calc.\ Var.\ Partial Differential Equations  
\newblock {\bf 18} (2003), 287--305.

\bibitem{Griffith:1921} 
{\sc A.~A.~Griffith}. 
\newblock {\em The phenomena of rupture and flow in solids}.
\newblock Philos.\ Trans.\ R.\ Soc.\ London 
\newblock {\bf 221} (1921), 163--198. 


\bibitem{vonKarman}
{\sc T.~von~K\'arm\'an}.
\newblock {\em Festigkeitsprobleme im Maschinenbau}.
\newblock In: Encyclop\"adie der Mathematischen Wissenschaften Vol.\ IV/4, pp.\ 311-385, Leipzig 1910.

\bibitem{Kirchhoff}
{\sc G.~Kirchhoff}.
\newblock {\em \"Uber das Gleichgewicht und die Bewegung einer elas\-ti\-schen Scheibe}.
\newblock J.\ Reine Angew.\ Math.\  
\newblock {\bf 40} (1850), 51--88.

\bibitem{LeDretRaoult-i}
{\sc H.~Le Dret, A.~Raoult}.
\newblock {\em The nonlinear membrane model as a variational limit of three-dimensional elasticity}.
\newblock J.\ Math.\ Pures Appl.\ 
\newblock {\bf 74} (1995), 549--578.


\bibitem{LeDretRaoult-ii}
{\sc H.~Le Dret, A.~Raoult}.
\newblock {\em The membrane shell model in nonlinear elasticity: a variational asymptotic derivation}.
\newblock J.\ Nonlinear Sci.\ 
\newblock {\bf 6} (1996), 59--84.

\bibitem{LewickaMahadevanPakzad}
{\sc M.~Lewicka, L.~Mahadevan, M.~R.~Pakzad} 
\newblock {\em The F{\"o}ppl-von K{\'a}rm{\'a}n equations for plates with incompatible strains}. 
\newblock Proc.\ R.\ Soc.\ Lond.\ Ser.\ A 
\newblock {\bf 467} (2011), 402--426. 

\bibitem{Love}
{\sc A.~E.~H.~Love}.
\newblock {\em A treatise on the mathematical theory of elasticity}.
\newblock Cambridge University Press, Cambridge 1927

\bibitem{NegriToader:2013}
{\sc M.~Negri, R.~Toader}.
\newblock {\em Scaling in fracture mechanics by Ba$\check z$ant's law: from finite to linearized elasticity}. 
\newblock Math.\ Models Methods Appl.\ Sci.\ 
\newblock {\bf 25} (2015), 1389--1420. 

\bibitem{Schmidt:06}
{\sc B.~Schmidt}. 
\newblock {\em A derivation of continuum nonlinear plate theory from atomistic models}. 
\newblock SIAM Multiscale Model.\ Simul.\
\newblock {\bf 5} (2006), 664--694. 

\bibitem{Schmidt:07}
{\sc B.~Schmidt}. 
\newblock {\em Plate theory for stressed heterogeneous multilayers of finite bending energy}. 
\newblock J.\ Math.\ Pures Appl.\ 
\newblock {\bf 88} (2007), 107--122. 


\end{thebibliography}
\end{document}